\documentclass[preprint,review,12pt]{elsarticle}

\usepackage[ruled]{algorithm2e}
\usepackage[a4paper,top=2.54cm,bottom=2.54cm,left=3.17cm,right=3.17cm,%
includehead,includefoot]{geometry}

\usepackage{amsmath,amssymb,amsfonts,amsthm}
\usepackage{graphicx}
\graphicspath{{figure/}} 
\usepackage{caption}
\usepackage{subfig}
\usepackage{float}
\usepackage{xcolor}
\usepackage{amssymb,url,color}
\usepackage{longtable}
\usepackage{multirow}
\usepackage{supertabular}
\usepackage{epstopdf}
\usepackage{listings}
\usepackage{epsfig}
\usepackage{algpseudocode}
\usepackage{amsmath}
\usepackage{bm}
\usepackage{amssymb}
\usepackage{amsthm}

\usepackage{mathrsfs}
\usepackage{dsfont}

\allowdisplaybreaks

\newtheorem{theorem}{Theorem}[section]

\newtheorem{definition}{Definition}[section]

\newcommand{\R}{\mathbb{R}}

\newcommand{\XC}{\mathcal{X}}
\newcommand{\LC}{\mathcal{L}}
\newcommand{\BC}{\mathcal{B}}
\newcommand{\ZC}{\mathcal{Z}}
\newcommand{\YC}{\mathcal{Y}}
\newcommand{\FC}{\mathcal{F}}
\newcommand{\MC}{\mathcal{M}}
\newcommand{\EC}{\mathcal{E}}
\newcommand{\NC}{\mathcal{N}}
\newcommand{\WC}{\mathcal{W}}
\newcommand{\TC}{\mathcal{T}}
\newcommand{\QC}{\mathcal{Q}}
\newcommand{\PC}{\mathcal{P}}
\newcommand{\SC}{\mathcal{S}}

\journal{Signal Processing}

\begin{document}

\begin{frontmatter}



\title{A Novel Transformed Fibered Rank Approximation with Total Variation	Regularization for Tensor Completion}


\author[1]{Ziming Chen}

\author[1]{Xiaoqing Zhang\thanks{Corresponding author: 52215500017@stu.ecnu.edu.cn}}


\address[1]{East China Normal University, Shanghai, People’s Republic of China}

\cortext[cor1]{Corresponding author: Xiaoqing Zhang}


\begin{abstract}
Recently, tensor fibered rank has demonstrated impressive performance by effectively leveraging the global low-rank
property in all directions for low-rank tensor completion (LRTC). 
However, it still has some limitations. 
Firstly, the typical tensor fibered rank approximation based on tensor nuclear norm (TNN) 
processes fixed and data-independent transformation, which may not be optimal for the underlying tensor structure. 
Secondly, it ignores the local piecewise smoothness of the dataset. 
To address these limitations, we present a nonconvex learnable transformed fibered nuclear norm (NLTFNN) model for LRTC,
which uses a learnable transformed fibered nuclear norm with Log-Determinant (LTFNNLog) as tensor fibered rank approximation, 
and employs a total variation (TV) regularization to explore local piecewise smoothness. 
An efficient algorithm based on the alternating direction method of multipliers (ADMM) is developed to solve NLTFNN and 
the convergence of the algorithm is proved theoretically. 
Experiments on various datasets show the superiority of NLTFNN over several existing methods.

\end{abstract}




\end{frontmatter}


\section{Introduction}
LRTC has found extensive application in various practical problems, such as image denoising \cite{GXC20}, subspace clustering \cite{FLY22}, multi-task learning \cite{ZYZ22}, and image inpainting \cite{QML20}.
It aims to reconstruct the original data by effectively exploiting the inherent redundancy present in incomplete data. Mathematically, the typical LRTC model can be expressed as follows:
\begin{equation}
	\label{lrtc}
	\begin{aligned}		
		\underset{\ZC} \min 
		&\quad {\rm rank}(\ZC) \\
		{\rm s.t.} 
		& \quad P_{\Omega}(\ZC) = P_{\Omega}(\BC)	,
	\end{aligned}
\end{equation}
where $\ZC$ is the underlying tensor, 
$\BC$ is the observed tensor,
$\Omega$ is the index set of known elements in $\BC$,
rank$(\cdot)$ is a rank of the tensor,
and $P_{\Omega}(\cdot)$ is a projection operator on $\Omega$.


In contrast to the unique definition of matrix rank, the definition of tensor rank is not unique, with different definitions arising from various tensor decompositions,
such as CANDECOMP/PARAFAC (CP) rank \cite{KB09}, Tucker rank \cite{TUC66}, and tensor tubal rank \cite{KMM11}.  
The CP-rank is defined in a simple manner, but its computation is NP-hard. 
And the Tucker rank is represented by a vector where the $i$th element denotes the rank of mode-$i$ unfolding matrix, which of big size.
In addition, tensor tubal rank based on tensor singular value decomposition (t-SVD) \cite{KBH13} only considers the low-rankness of mode-$3$.
Therefore, Zheng et al. \cite{ZHZ20} proposed tensor fibered rank by 
extending tubal rank and generalizing t-SVD, which provides a simultaneous and flexible representation of the correlations along each mode.
Then, 3DTNN and 3DLogTNN were also proposed as the convex and nonconvex relaxation of tensor fibered rank, respectively.

However, current tensor recovery methods, which are based on tensor fibered rank, predominantly depend on the t-SVD framework utilizing the Discrete Fourier Transform (DFT). 
It may not conducive to satisfying the low-rank property for tensor dataset \cite{WGF22}. 
In order to achieve an optimal low-rank representation, Xu et al. \cite{XZN19} replace the DFT with the Discrete Cosine Transform (DCT),
and Jiang et al. \cite{JNZ20} suggest the semi-invertible framelet transform to represent the transformation.
Unlike the pre-defined transforms mentioned above, Jiang et al. \cite{JZZ21} proposed learnable transform, which is updated during the course of the algorithmic iteration. 
Then, Wu et al. \cite{WGF22} proposed a LRTC model that utilizes multi-dimensional learnable transforms to fully exploit the low-rankness. 
Moreover, the inflexible treatment of singular values in the convex tensor fibered rank approximation is analogous to that of TNN.
To address this shortcoming, nonconvex surrogates, including the Log-Determinant function \cite{JHZ17}, the Schatten-$p$ function \cite{GF20}, and the Laplace function \cite{XZJ19}, are frequently employed as replacements for TNN in LRTC models. 
Building upon the aforementioned research, we introduce a novel method for tensor fibered rank approximation, namely a  learnable transformed fibered nuclear norm with Log-Determinant (LTFNNLog), which combines learnable transforms and nonconvex function. 
The proposed LTFNNLog enables simultaneous representation of the inter-mode correlations while achieving enhanced accuracy in the low-rankness of each mode.

In addition, local smoothness is a crucial attribute of tensor datasets, such as image and video.
Therefore, it is insufficient to consider only the low-rank property when solving the LRTC problem.
Liu et al. \cite{LHS15} employed TV regulariation in tensor recovery to preserve the piecewise smoothness.
For the same purpose, we apply TV regularization to the proposed model, thereby enhancing the internal smoothness of the data and improving the effectiveness of tensor completion.

The main contributions of this paper are illustrated as follows.
\begin{enumerate}
	\item 
	 
	We introduce a nonconvex learnable transformed fibered nuclear norm model for LRTC, named NLTFNN. 
	In this novel model, LTFNNLog is incorporated to approximate tensor fibered rank and effectively explore the low-rank structure of three modes in tensors,
	and TV regularization is also integrated to preserve the inherent piecewise smoothness.

	\item A framework for solving the new model is established based on ADMM.
	It has been demonstrated that, under specific conditions, the generated sequences converge to the Karush-Kuhn-Tucker (KKT) point.
	
	
\end{enumerate}

The remainder of this paper is organized as follows.
In section 2, we introduce the fundamental symbols and related definitions of tensors.
In Section 3, we present the NLTFNN model and solve it by an effective algorithm.
In Section 4, the convergence of the algorithm is demonstrated under specific conditions.
Experimental results that illustrate the feasibility and effectiveness of NLTFNN model are present in Section 5. 
Finally, we conclude this paper with a brief summary.

\section{Notations and Preliminaries}

We use the calligraphic font $\XC$, 
uppercase characters $X$ 
and lowercase letters $x$ to denote tensors, 
matrices and scalars, respectively. 
The domain of real numbers is denoted as $\mathbb{R}$. 
For a 3-order tensors $\XC = [a_{ijk}] \in \mathbb{R}^{n_1\times n_2\times n_3}$, 
we use $X_1^{(i)}\in\R^{n_2\times n_3}$,
$X_2^{(i)}\in\R^{n_3\times n_1}$ and
$X_3^{(i)}\in\R^{n_1\times n_2}$ to denote the $i$th
mode-$1$, mode-$2$ and mode-$3$
slices of $\XC$, respectively.
The Frobenius norm of $\XC$ is defined as 
$\|\XC\|_F = \sqrt{\sum_{i,j,k}|a_{ijk}|^2},$ and
the $l_1$ norm of $\XC$ is defined as 
$\|\XC\|_1 = \sum_{i,j,k}|a_{ijk}|$.
The notation $\hat{\XC}_i$ means the DFT along mode-$i$ of $\XC$,
i.e., $\hat{\XC}_i = {\rm fft}(\XC,[\,],i)$.
Naturally, we have $\XC = {\rm ifft}(\hat{\XC}_i,[\,],i)$.

For an $N$-order tensor $\XC \in \R ^{n_1\times n_2\times\cdots\times n_{N}}$,
the mode-$i$ unfolding operator performing on $\XC$
is denoted as
$X_{(i)}= {\rm unfold}(\XC,i)\in
\R^{n_i\times \prod_{k\neq i}n_k},$ 
and its inversion is denoted by
$\XC = {\rm fold}({\rm unfold}(\XC,i),i).$

\begin{definition}(Mode-$i$ Product \cite{LLL22})
	The mode-$i$ product between a tensor $\XC\in
	\R^{n_1\times n_2\times\cdots\times n_{N}}$ and a matrix
	$L \in \R^{J\times n_i}$ is defined as
	\begin{equation*}
		\XC_{L,i} = \XC\times_{i} L
		= {\rm fold}(LX_{(i)},i)\in
		\R^{n_1\times\cdots\times n_{i-1}\times J\times n_{i+1}\times\cdots\times n_{N}}.
	\end{equation*}
\end{definition}

\begin{definition}(Mode-$i$ t-SVD \cite{ZHZ20})
	A 3-order tensor $\XC\in\R^{n_1\times n_2\times n_3}$ has
	the t-SVD along mode-$i$, i.e.,
	\begin{equation*}
		\XC = \mathcal{U}_i \ast_i \SC_i \ast_i
		\mathcal{V}_i^{T_i},
		\quad i=1,2,3,
	\end{equation*}
	where $\mathcal{U}_i$ and $\mathcal{V}_i$ are the mode-$i$
	orthogonal tensors, $\SC_i$ is the mode-$i$ diagonal tensor, and $T_i$ means the mode-$i$ transpose.
\end{definition}

\begin{definition}(Mode-$i$ TNN \cite{ZHZ20})
	The mode-$i$ TNN of a 3-order tensor
	$\XC\in\R^{n_1\times n_2\times n_3}$ is defined as 
	the sum of singular values of all the mode-$i$
	slices of tensor $\hat{\XC}_i$, i.e.,
	\begin{equation*}
		\|\XC\|_{TNN_i} = \sum_{k=1}^{n_i}
		\|(\hat{X}_i)_i^{(k)}\|_* .
	\end{equation*}
\end{definition}

%

For a 3-order tensor $\XC\in\R^{n_1\times n_2\times n_3}$
and a matrix $L\in \R^{J\times n_i}$, 
we have the decomposition $\XC_{L,i} = \mathcal{U}_i \ast_i \SC_i \ast_i \mathcal{V}_i^{T_i}.$
Then, we use the symbol ${\rm rank}_{f_i,L}(\XC)$ to denote the mode-$i$ transformed fibered rank of $\XC$ 
with the transformed matrix $L$, 
which defined as the number of nonzero mode-$i$ fibers of 
$\SC_i$.
And	the transformed fibered rank of  $\XC$ with $L$
is a vector denoted as ${\rm rank}_{f,L}(\XC),$ where
\begin{equation*}
	{\rm rank}_{f,L}(\XC)=({\rm rank}_{f_1,L}(\XC),
	{\rm rank}_{f_2,L}(\XC),{\rm rank}_{f_3,L}(\XC)).
\end{equation*}



In order to give the
approximation of transformed fibered rank,
we propose transformed fibered nuclear norm (TFNN).
With the transformed matrices $L_i\in \R^{J\times n_i}, i=1,2,3,$ 
the TFNN of
$\XC\in\R^{n_1\times n_2\times n_3}$ is defined as:
\begin{equation*}
	\|\XC\|_{TFNN} = \sum_{i=1}^{3} a_i 
	\|\XC_{L_i,i}\|_{TNN_i},
\end{equation*}
where $a_i\geq 0 $ and $\sum_{i=1}^{3}a_i =1.$



In fact, TFNN is the convex relaxation of transformed fibered rank, and inevitably constrains within the framework of TNN.
In the real world, singular values have obvious physical meaning and need to be treated differently, but TNN simply sums up all singular values.
To overcome this limitation, we take advantage of the nonconvex function and propose a nonconvex approximation called the
TFNN based on Log-Determinant (TFNNLog).
Firstly, we give the definition of mode-$i$ LogTNN 
of $\XC$ as follows:
\begin{align*}
	\Psi_i(\XC) 
	&= \sum_{j=1}^{n_i}
	\sum_{k=1}^{m_i} {\rm log}
	\left(1+\sigma_k^2\left( (\hat{X}_i)_i^{(j)} \right)\right),
\end{align*}
where $\sigma_k(X)$ is the $k$th singular value of $X$ 
and $m_i = \min\{ \{n_1,n_2,n_3\}\textbackslash \{n_i\} \}$.
Then, we propose TFNNLog of $\XC$ 
and $L_i\in \R^{J\times n_i}, i=1,2,3,$ i.e.,
\begin{equation*}
	\|\XC\|_{TFNNLog} = \sum_{i=1}^{3} a_i
	\Psi_i(\XC_{L_i,i}),
\end{equation*}
where $a_i\geq 0 $ and $\sum_{i=1}^{3}a_i =1.$


%


\section{NLTFNN Model And Solving Algorithm}

\subsection{Proposed NLTFNN Model}

In order to exploit the inherent low-rank property of the data along every mode, we employ the learnable transformed fibered nuclear norm with Log-Determinant, that is, the LTFNNLog, as a surrogate of tensor rank as follows:
\begin{equation}
	\label{proposed0}
	\begin{aligned}		
		\underset{\mathcal{Z},L_i}\min &\quad
		\|\ZC\|_{TFNNLog} \\
		{\rm s.t.} 	&\quad   L_iL_i^T=I_i,\\
		& \quad P_{\Omega}(\ZC)=P_{\Omega}(\BC)	,
	\end{aligned}
\end{equation}
where $L_i, i=1,2,3,$ are the learnable transformed
matrices of mode-$i$. 
%
%

In addition, we add TV regularization in \eqref{proposed0}, as it enables us to exploit local smoothness in the data. 
Consequently, our proposed NLTFNN model can be formulated as
follows:
\begin{equation}
	\label{proposed}
	\begin{aligned}		
		\underset{\mathcal{Z},L_i}\min &\quad
		\|\ZC\|_{TFNNLog} + \tau\|\ZC\|_{TV} \\
		{\rm s.t.} 	&\quad   L_iL_i^T=I_i,\\
		& \quad P_{\Omega}(\ZC)=P_{\Omega}(\BC)	,
	\end{aligned}
\end{equation}
where $\tau$ is regularization parameter,
and the TV regularization for a 3-order tensor is defined as
\begin{equation*}
	\|\ZC\|_{TV}=\|D\ZC\|_1=\|D_h\ZC\|_1+\|D_v\ZC\|_1+\|D_t\ZC\|_1,
\end{equation*}
where $D=[D_h,D_v,D_t]$ is the three-dimensional 
difference operator, and $D_h,D_v,D_t$ denote the 
first-order forward finite difference operators along 
the horizontal, vertical, and tubal direction, respectively.
More details about TV regularization can be found in \cite{QBN21}.


\subsection{Solving Algorithm Based on ADMM Framework}

In order to solve model \eqref{proposed} 
within the framework of ADMM, we denote $\mathcal{O}$ as a tensor with all entries being $0$, and
perform the equivalent form of \eqref{proposed}
as follows:

\begin{equation}
	\label{proposed_eq}
	\begin{aligned}		
		\underset{\mathcal{Z},\XC_i,\YC_i,L_i,\FC,\MC,\EC}\min &\quad\sum_{i=1}^3 a_i \Psi_i(\XC_i)+\tau\|\FC\|_{1} \\
		{\rm s.t.} 	&\quad  \YC_i=\ZC,
		\quad  \XC\times_i L_i^T = \YC_i,\\
		&\quad  \FC=D\MC,
		\quad  \MC=\ZC,\\	
		&\quad  \ZC+\EC=\BC,
		\quad P_{\Omega}(\EC)=\mathcal{O},	\\
		&\quad  L_iL_i^T =I_i.
	\end{aligned}
\end{equation}

The augmented Lagrangian function of \eqref{proposed_eq} is
\begin{equation}
	\label{Lgr}
	\begin{aligned}
		& \quad  \mathscr{L}_{\mu}(\ZC,\FC,\MC,\EC,
		\{\XC_i,\YC_i,L_i\}_{i=1}^3) \\
		& = \sum_{i=1}^3 \left( a_i \Psi_i(\XC_i)
		+ \mathds{1}_{L_i L_i^T = I_i} \right)
		+ \tau \|\FC\|_{1} 
		+ \mathds{1}_{P_{\Omega}(\EC)=\mathcal{O}} 	\\
		& \quad + \sum_{i=1}^3 \left(\langle \XC_i \times_i
		L_i^T - \YC_i, \NC_i \rangle 
		+ \langle \YC_i - \ZC, \WC_i \rangle \right) \\
		& \quad + \langle \FC - D\MC, \TC \rangle 
		+ \langle \MC - \ZC, \QC \rangle
		+ \langle \ZC + \EC - \BC, \PC\rangle \\
		& \quad + \frac{\mu}{2} \sum_{i=1}^3 \left(
		\left\| \XC_i \times_i L_i^T -\YC_i \right\|_F^2 
		+ \left\| \YC_i - \ZC \right\|_F^2 \right) \\
		& \quad + \frac{\mu}{2} \left( 
		\left\| \FC - D\MC \right\|_F^2 
		+ \left\| \MC - \ZC \right\|_F^2 
		+ \left\| \ZC + \EC - \BC \right\|_F^2 \right)	,
	\end{aligned}
\end{equation}
where $\mathds{1}$ denotes the indicator function, $\{\NC_i,\WC_i\}_{i=1}^3,\TC,\QC,\PC$ are the Lagrangian
multipliers, and $\mu$ is the penalty parameter.

$\bullet$ \textbf{Updating $\{\YC_i\}_{i=1}^3$ subproblem}
\begin{equation*}
	\label{subY}
	\YC_i^{p+1}=\arg\underset{\YC_i}\min\left\{ \frac{\mu^p}{2}\left(\left\|\XC_i^{p}\times_i ({L_i^p})^T-\YC_i+\frac{\NC_i^p}{\mu^p}\right\|_F^2+\left\|\YC_i-\ZC^p+\frac{\WC_i^p}{\mu^p}\right\|_F^2\right) \right\}.
\end{equation*}
This quadratic objective function is differentiable and convex. 
We can get the minimal solution by deriving the first-order optimal condition as follows:
\begin{equation}
	\label{soluY}
	\YC_i^{p+1}=\frac{1}{2}\left(\XC_i^{p}\times_i({L_i^p})^T+\ZC^p+\frac{\NC_i^p-\WC_i^p}{\mu^p} \right).
\end{equation}

$\bullet$ \textbf{Updating $\{\XC_i\}_{i=1}^3$ subproblem}
\begin{align}
	\label{subX}
	\XC_i^{p+1}&=\arg\underset{\XC_i}\min\left\{ 
	a_i \Psi_i(\XC_i)+\frac{\mu^p}{2}\left\|
	\XC_i\times_i ({L_i^p})^T-\YC^{p+1}_i
	+\frac{\NC^p_i}{\mu^p}\right\|_F^2\right\}.
\end{align}
By Lemma 1 in \cite{LZJ22} and the LogDeterminant minimization in \cite{YLL22}, 
$\XC_i$ subproblem can be solved as 
\begin{equation}
	\label{soluX_2}
	\XC_i^{p+1} ={\rm prox}_{{\rm Log},\frac{a_i}{\mu^p}}
	\left(\left(\YC^{p+1}_i-\frac{\NC^p_i}{\mu^p}\right)
	\times_i L_i^p,i \right).
\end{equation}

$\bullet$ \textbf{Updating $\FC$ subproblem}	
\begin{equation*}
	\label{subF}
	\FC^{p+1}=\arg\underset{\FC}\min\left\{\tau\|\FC\|_1+\frac{\mu^p}{2}\left\|\FC-D\MC^p+\frac{\TC^p}{\mu^p}\right\|_F^2\right\}.
\end{equation*}	
By the soft-shrinkage operator, 
$\FC$ subproblem can be written as
\begin{equation}
	\label{soluF}
	\mathcal{F}^{p+1}=\operatorname{sgn}\left(D\MC^p-\frac{\TC^p}{\mu^p}\right) \circ \max \left\{\left|D\MC^p-\frac{\TC^p}{\mu^p}\right|-\frac{\tau}{\mu^p}, 0\right\},
\end{equation}
where sgn is a point-wise operator, $\operatorname{sgn}(x)$ denotes the sign of $x \neq 0$ and $\operatorname{sgn}(0)=0$, and $\circ$ denotes point-wise product.

$\bullet$ \textbf{Updating $\MC$ subproblem}
\begin{equation*}
	\label{subM}
	\MC^{p+1}=\arg\underset{\MC}\min\left\{\frac{\mu^p}{2}\left(\left\|\FC^{p+1}-D\MC+\frac{\TC^p}{\mu^p}\right\|_F^2+\left\|\MC-\ZC^p+\frac{\QC^p}{\mu^p}\right\|_F^2\right)\right\}.
\end{equation*}
For $\mathcal{M}$ subproblem, it equals to solve the following equation:
\begin{equation}
	\label{subM2}
	\mu^p\left(\mathcal{I}+D^* D\right) \mathcal{M}=D^*\left(\mu^p\FC^{p+1}+\TC^p\right)+\mu^p\ZC^p-\QC^p,
\end{equation}
where $D^*$ indicates the adjoint operator of $D$. By the block circulant structure of the matrix corresponding to the operator $D^* D$, 
it can be diagonalized by the three-dimensional fast Fourier transform (3DFFT) matrix. Then \eqref{subM2} can be solved explicitly by
\begin{equation}
	\label{soluM}
	\mathcal{M}^{p+1}=\operatorname{iFFT} 3\left(\frac{\operatorname{FFT} 3\left( D^*\left(\mu^p\FC^{p+1}+\TC^p\right)+\mu^p\ZC^p-\QC^p \right)}{\mu^p \mathcal{K}+\mu^p\left|\operatorname{FFT} 3\left(D_h\right)\right|^2+\mu^p\left|\operatorname{FFT} 3\left(D_v\right)\right|^2+\mu^p\left|\operatorname{FFT} 3\left(D_t\right)\right|^2}\right),
\end{equation}
where FFT3 and iFFT3 denote the 3DFFT and its inverse transform, respectively. $\mathcal{K}$ denotes a tensor with all entries being 1, and the division and square are performed in a point-wise manner.

$\bullet$ \textbf{Updating $\{{\rm \LC_i}\}_{i=1}^3$ subproblem}
\begin{equation}
	\label{subL}
	\begin{aligned}
		L_i^{p+1}
		& = \arg\underset{L_i} \min \left\{
		\frac{\mu^p}{2} \left\| \XC_i^{p+1} \times_i
		L_i^T - \YC_i^{p+1} + \frac{\NC_i^p}{\mu^p} 
		\right\|_F^2 + \mathds{1}_{L_iL_i^T =I_i}
		\right\} \\
		& = \arg\underset{L_i L_i^T=I_i} \min 
		\left\| L_{i}^{T} \XC_{i(i)}^{p+1} 
		- \left( \YC_i^{p+1} - \frac{\NC_i^p}{\mu^p} 
		\right)_{(i)} \right\|_F^2\\	
		& = \arg\underset{L_i L_i^T=I_i} \min
		{\rm Tr} \left[ \left( L_i^{T} \XC_{i(i)}^{p+1} 
		- \left( \YC_i^{p+1} - \frac{\NC_i^p}{\mu^p} 
		\right)_{(i)} \right)^T \right. \\
		& \hspace{11em} \cdot \left. \left(
		L_i^{T}\XC_{i(i)}^{p+1} - \left( \YC_i^{p+1}
		- \frac{\NC_i^p}{\mu^p} \right)_{(i)} 
		\right) \right]\\
		& = \arg\underset{L_i L_i^T=I_i} \max
		{\rm Tr}
		\left[ \left( \YC_i^{p+1} 
		- \frac{\NC_i^p}{\mu^p} \right)
		{\XC_{i(i)}^{p+1}}^T L_{i} \right].
	\end{aligned}		
\end{equation}
Supposing the matrix SVD of $\left(\YC_i^{p+1}-\frac{\NC_i^p}{\mu^p}\right){\XC_{i(i)}^{p+1}}^T$ is $U_iS_iV_i^T$, we have
\begin{equation*}
	{\rm Tr}(U_iS_iV_i^TL_i)={\rm Tr}(S_iU_iV_i^TL_i).
\end{equation*}
Since $S_i$ is a diagonal matrix, the maximization problem in \eqref{subL} can be maximized when the diagonal elements of $U_iV_i^TL_i$ is positive maximum. By the Cauchy-Schwarz inequality, it can be achieved when $L_i=(U_iV_i^T)^T$, in which case the diagonal elements are all 1. Hence the closed-form solution of \eqref{subL} is
\begin{equation}
	\label{soluL}
	L_i^{p+1}=V_iU_i^T.
\end{equation}

$\bullet$ \textbf{Updating $\ZC$ subproblem}
\begin{equation*}
	\begin{aligned}
	\ZC^{p+1}=\arg\underset{\ZC}\min&\left\{
	\sum_{i=1}^3\left\|\YC^{p+1}_i-\ZC+\frac{\WC^p_i}{\mu^p}\right\|_F^2 
	+\left\|\MC^{p+1}-\ZC+\frac{\QC^p}{\mu^p}\right\|_F^2 \right. \\
	& \hspace{7em}
	\left. +\left\|\ZC+\EC^p-\BC+\frac{\PC^p}{\mu^p}\right\|_F^2\right\}.
	\end{aligned}
\end{equation*}
This quadratic objective function can be slove as follows,
\begin{equation}
	\label{soluZ}
	\ZC^{p+1}=\frac{1}{5}\left( \sum_{i=1}^3\left(\YC_i^{p+1}+\frac{\WC^p_i}{\mu^p}\right)+\MC+\BC-\EC^p+\frac{\QC^p-\PC^p}{\mu^p} \right).
\end{equation}

$\bullet$ \textbf{Updating $\EC$ subproblem}
\begin{equation*}
	\label{subE}
	\EC^{p+1}=\arg\underset{\EC}\min\left\{\mathds{1}_{P_{\Omega}(\EC)=\mathcal{O}}+\frac{\mu^p}{2}\left\|\ZC^{p+1}+\EC-\BC+\frac{\PC^p}{\mu^p}\right\|_F^2\right\}.
\end{equation*}
Similarly, the $\EC$ subproblem can be solved as
\begin{equation}
	\label{soluE}
	\EC^{p+1}=P_{\Omega^C}\left(\BC-\ZC^{p+1}+\frac{\PC^p}{\mu^p}\right)+P_{\Omega}(\mathcal{O}),
\end{equation}
where $\Omega^C$ is the complementary of the index set $\Omega$.

$\bullet$ \textbf{Updating the multipliers}
\begin{equation}
	\label{multiplier}
	\begin{aligned}
		&\NC_i^{p+1}=\NC_i^p+\mu^p(\XC_i^{p+1}\times_i({L_i^{p+1}})^T-\YC_i^{p+1}),i=1,2,3,\\
		&\WC_i^{p+1}=\WC_i^p+\mu^p(\YC^{p+1}_i-\ZC^{p+1}),i=1,2,3,\\
		&\TC^{p+1}=\TC^p+\mu^p(\FC^{p+1}-D\MC^{p+1}),\\
		&\QC^{p+1}=\QC^p+\mu^p(\MC^{p+1}-\ZC^{p+1}),\\
		&\PC^{p+1}=\PC^p+\mu^p(\ZC^{p+1}+\EC^{p+1}-\BC).
	\end{aligned}
\end{equation}

The proposed ADMM for solving \eqref{proposed_eq} is 
summerized in Algorithm \ref{Alg}.

\begin{algorithm}
	\caption{ADMM Framework of NLTFNN Model}
	\label{Alg}
	\KwIn{The observed tensor $\BC$, 
		maximum iteration $\eta$ , stopping criterion $\epsilon$ ,
		$a=[a_1,a_2,a_3], \tau, \mu^0, \mu_{max}$ and $\rho$.} 
	\bf{Initialize:} $ L_i^0=I_i,  \PC^0=\MC^0=\QC^0=\EC^0=\mathcal{O},\ZC^0=\BC,$
	$\XC_i^0=\YC_i^0=\NC_i^0=\WC_i^0=\mathcal{O},$ $ \FC^0=\TC^0=\mathcal{O}$, $p=1$.\\
	\While{\rm not converged and $p\leq \eta$}{
		\For{$i=1:3$}{
			{\rm Update $\YC_i^p$ via} \eqref{soluY}. \\
			{\rm Update $\XC_i^p$ via} \eqref{soluX_2}. \\
			{\rm Update $L_i^p$ via} \eqref{soluL}. \\
		}
		{\rm Update $\FC^p$, $\MC^p$, $\ZC^p$ and $\EC^p$ via \eqref{soluF}, \eqref{soluM}, \eqref{soluZ} and \eqref{soluE}, respectively}. \\
		{\rm Update the multipliers $\{\NC_i^p\},\{\WC_i^p\},\TC^p,\QC^p,\PC^p$ via} \eqref{multiplier}. \\
		{\rm Update} $\mu^p=\min\{\mu_{max},\rho\mu^p\}$. \\
		{\rm  Check the convergence conditions: $\|\ZC^p-\ZC^{p-1}\|_{\infty}<\epsilon$}. \\
		$p=p+1$. \\
	}
	\KwOut{\rm The completed tensor $\ZC$.}
\end{algorithm}

\section{Convergence analysis}

It should be noted that the NLTFNN model is nonconvex and  the convergence analysis of the ADMM for nonconvex problems is limited \cite{CGW17}.
Consequently, the global convergence analysis we proposed requires certain assumptions about the sequences generated by
Algorithm \ref{Alg}.
And the next section presents evidence that the experimental convergence of Algorithm \ref{Alg} is strong.

\begin{theorem}
	Assume that the sequences $\{\{\NC_i^p\}_{i=1}^3,\{\WC_i^p\}_{i=1}^3,\QC^p,\TC^p,\PC^p\}_{p=1}^\infty$ of multipliers are bounded, we have\\
	$(a)$  $\{\{\YC_i^p\}_{i=1}^3,\{\XC_i^p\}_{i=1}^3,\{L_i^p\}_{i=1}^3,\ZC^p,\FC^p,\MC^p,\EC^p\}_{p=1}^\infty$ are bounded.\\
	$(b)$  $\{\{\YC_i^p\}_{i=1}^3,\{\XC_i^p\}_{i=1}^3,\{L_i^p\}_{i=1}^3,\ZC^p,\FC^p,\MC^p,\EC^p\}_{p=1}^\infty$ are Cauchy
	sequences.
	
\end{theorem}

\begin{proof}
	From  the augmented Lagrangian function \eqref{Lgr} and combining \eqref{multiplier} with simple calculations, it follows that
	\begin{align}
		\label{converge1}
		& \quad \mathscr{L}_{\mu^{p}}
		(\{\YC_i^{p}\}_{i=1}^3, 
		\{\XC_i^{p}\}_{i=1}^3, \{L_i^{p}\}_{i=1}^3, \ZC^{p}, \FC^{p}, \MC^{p}, \EC^{p}, 
		\{\NC_i^{p}\}_{i=1}^3,
		\{\WC_i^{p}\}_{i=1}^3,
		\QC^{p}, \TC^{p}, \PC^{p})  \notag \\
		& = \mathscr{L}_{\mu^{p-1}}
		(\{\YC_i^{p}\}_{i=1}^3, 
		\{\XC_i^{p}\}_{i=1}^3, \{L_i^{p}\}_{i=1}^3,
		\ZC^{p}, \FC^{p}, \MC^{p}, \EC^{p}, 
		\{\NC_i^{p-1}\}_{i=1}^3, 
		\{\WC_i^{p-1}\}_{i=1}^3, \notag  \\
		& \hspace{26.5em} \QC^{p-1}, 
		\TC^{p-1}, \PC^{p-1}) \notag  \\
		& \quad + \sum_{i=1}^3 \left(
		\langle \XC_i^{p} \times_i {L_i^{p}}^T 
		- \YC_i^{p}, \NC_i^{p} - \NC^{p-1}_i \rangle
		+ \frac{\mu^{p} - \mu^{p-1}}{2}
		\left\| \XC_i^{p} \times_i {L_i^{p}}^T 
		- \YC_i^{p} \right\|_F^2 \right. \notag \\
		& \quad + \left. \langle \YC_i^{p} 
		- \ZC^{p}, \WC^{p}_i - \WC_i^{p-1} \rangle
		+ \frac{\mu^{p} - \mu^{p-1}}{2} 
		\left\| \YC_i^{p} - \ZC^{p} \right\|_F^2
		\right)  \notag \\	
		& \quad + \langle \FC^{p} - D\MC^{p}, 
		\TC^{p} - \TC^{p-1} \rangle
		+ \frac{\mu^{p} - \mu^{p-1}}{2}
		\left\| \FC^{p} - D\MC^{p} \right\|_F^2  \notag \\
		& \quad + \langle \MC^{p} - \ZC^{p}, 
		\QC^{p} - \QC^{p-1} \rangle
		+\frac{\mu^{p}-\mu^{p-1}}{2} \left\| 
		\MC^{p} - \ZC^{p} \right\|_F^2 \notag  \\
		& \quad + \langle \ZC^{p} + \EC^{p} - \BC,
		\PC^{p} - \PC^{p-1}\rangle 
		+ \frac{\mu^{p} - \mu^{p-1}}{2} 
		\left\| \ZC^{p} + \EC^{p} 
		- \BC \right\|_F^2  \notag \\
		& = \mathscr{L}_{\mu^{p-1}} 
		(\{\YC_i^{p}\}_{i=1}^3,
		\{\XC_i^{p}\}_{i=1}^3, \{L_i^{p}\}_{i=1}^3,
		\ZC^{p}, \FC^{p}, \MC^{p}, \EC^{p},
		\{\NC_i^{p-1}\}_{i=1}^3,
		\{\WC_i^{p-1}\}_{i=1}^3,  \notag \\
		& \hspace{26.5em} \QC^{p-1}, 
		\TC^{p-1}, \PC^{p-1}) \notag \\
		& \quad + 
		\frac{\mu^p + \mu^{p-1}}{2(\mu^{p-1})^2}
		\sum_{i=1}^3 \left( 
		\|\NC_i^p - \NC_i^{p-1} \|_F^2
		+ \|\WC_i^p - \WC_i^{p-1}\|_F^2 \right) \notag \\
		& \quad + 
		\frac{\mu^p + \mu^{p-1}}{2(\mu^{p-1})^2}
		\left(\|\TC^p - \TC^{p-1}\|_F^2
		+ \|\QC^p - \QC^{p-1}\|_F^2
		+ \|\PC^p - \PC^{p-1}\|_F^2 \right). 
	\end{align}

	Since $\{\YC_i^{p+1}\},\{\XC_i^{p+1}\},\{L_i^{p+1}\},\ZC^{p+1},\FC^{p+1},\MC^{p+1},\EC^{p+1}$ are the minimizers of each subproblem, then we have
	\begin{equation*}
		\begin{aligned}
			& \quad \mathscr{L}_{\mu^{p}}
			(\{\YC_i^{p+1}\}_{i=1}^3,
			\{\XC_i^{p+1}\}_{i=1}^3,
			\{L_i^{p+1}\}_{i=1}^3,
			\ZC^{p+1}, \FC^{p+1}, \MC^{p+1}, \EC^{p+1}, \\
			& \hspace{20em} \{\NC_i^{p}\}_{i=1}^3,
			\{\WC_i^{p}\}_{i=1}^3,
			\QC^{p}, \TC^{p}, \PC^{p})\\
			& \leq \mathscr{L}_{\mu^{p}}
			(\{\YC_i^{p}\}_{i=1}^3,
			\{\XC_i^{p+1}\}_{i=1}^3,
			\{L_i^{p+1}\}_{i=1}^3,
			\ZC^{p+1}, \FC^{p+1}, \MC^{p+1}, \EC^{p+1}, \\
			& \hspace{20em} \{\NC_i^{p}\}_{i=1}^3,
			\{\WC_i^{p}\}_{i=1}^3,
			\QC^{p}, \TC^{p}, \PC^{p})\\
			& \leq \mathscr{L}_{\mu^{p}}
			(\{\YC_i^{p}\}_{i=1}^3,
			\{\XC_i^{p}\}_{i=1}^3,
			\{L_i^{p+1}\}_{i=1}^3,
			\ZC^{p+1}, \FC^{p+1}, \MC^{p+1}, \EC^{p+1}, \\
			& \hspace{20em} \{\NC_i^{p}\}_{i=1}^3,
			\{\WC_i^{p}\}_{i=1}^3,
			\QC^{p}, \TC^{p}, \PC^{p}) \\
			& \leq \cdots\\
			& \leq \mathscr{L}_{\mu^{p}}
			(\{\YC_i^{p}\}_{i=1}^3,
			\{\XC_i^{p}\}_{i=1}^3,
			\{L_i^{p}\}_{i=1}^3,
			\ZC^{p}, \FC^{p}, \MC^{p}, \EC^{p}, \\
			& \hspace{20em} \{\NC_i^{p}\}_{i=1}^3,
			\{\WC_i^{p}\}_{i=1}^3,
			\QC^{p}, \TC^{p}, \PC^{p}) .\\
		\end{aligned}
	\end{equation*}
	
	Combining the above inequality with \eqref{converge1} and then iterating the inequality chain for $p$ times, we obtain
	\begin{align}
		\label{converge2}
		& \quad \mathscr{L}_{\mu^{p}}
		(\{\YC_i^{p+1}\}_{i=1}^3,
		\{\XC_i^{p+1}\}_{i=1}^3,
		\{L_i^{p+1}\}_{i=1}^3,
		\ZC^{p+1}, \FC^{p+1}, \MC^{p+1}, \EC^{p+1}, \notag\\
		& \hspace{20em} \{\NC_i^{p}\}_{i=1}^3,
		\{\WC_i^{p}\}_{i=1}^3,
		\QC^{p}, \TC^{p}, \PC^{p}) \notag\\
		& \leq \mathscr{L}_{\mu^{0}}
		(\{\YC_i^{1}\}_{i=1}^3,
		\{\XC_i^{1}\}_{i=1}^3,
		\{L_i^{1}\}_{i=1}^3,
		\ZC^{1}, \FC^{1}, \MC^{1}, \EC^{1}, \notag\\
		& \hspace{20em}  \{\NC_i^{0}\}_{i=1}^3,
		\{\WC_i^{0}\}_{i=1}^3,
		\QC^{0}, \TC^{0}, \PC^{0}) \notag\\
		& \quad + \sum_{n=1}^{p} 
		\frac{\mu^n + \mu^{n-1}}{2(\mu^{n-1})^2}
		\sum_{i=1}^3 \left(\|\NC_i^n - \NC_i^{n-1}\|_F^2
		+ \|\WC_i^n - \WC_i^{n-1}\|_F^2 \right) \notag\\
		& \quad + \sum_{n=1}^{p}
		\frac{\mu^n + \mu^{n-1}}{2(\mu^{n-1})^2}
		\left(\|\TC^n - \TC^{n-1}\|_F^2
		+ \|\QC^n - \QC^{n-1}\|_F^2
		+ \|\PC^n - \PC^{n-1}\|_F^2 \right).		
	\end{align}
	According to the boundedness of the multiplier sequences, for any $n$, we set $c$ is the upper bound of
	\begin{equation*}
		\{\|\NC_i^n-\NC_i^{n-1}\|_F^2,\|\WC_i^n-\WC_i^{n-1}\|_F^2,\|\TC^n-\TC^{n-1}\|_F^2,\|\QC^n-\QC^{n-1}\|_F^2,\|\PC^n-\PC^{n-1}\|_F^2\}.
	\end{equation*}
	Clearly, $c$ is bounded. 
	
	Next, because $\mu^{n+1}=\rho\mu^n$, where $\rho>1$, we arrive at
	\begin{equation}
		\label{converge3}
		\sum_{n=1}^{\infty}
		\frac{\mu^n + \mu^{n-1}}{2(\mu^{n-1})^2}
		= \sum_{n=1}^{\infty} \frac{1 + \rho}{2\mu^{n-1}}
		= \frac{\rho(1 + \rho)}{2\mu^0(\rho-1)} < \infty.
	\end{equation}
	Thus we get that the augmented Lagrangian function on the left side of \eqref{converge2} is bounded. By adding several terms, we get
	\begin{equation*}
		\begin{aligned}
			& \quad \mathscr{L}_{\mu^{p}} (\{\YC_i^{p+1}\}_{i=1}^3,
			\{\XC_i^{p+1}\}_{i=1}^3,
			\{L_i^{p+1}\}_{i=1}^3,
			\ZC^{p+1}, \FC^{p+1}, \MC^{p+1}, \EC^{p+1},
			\{\NC_i^{p}\}_{i=1}^3,\\
			& \hspace{23em} \{\WC_i^{p}\}_{i=1}^3,
			\QC^{p}, \TC^{p}, \PC^{p}) \\
			& \quad + \frac{1}{2\mu^p} \left( \sum_{i=1}^3
			\left(\|\NC_i^p\|_F^2 + \|\WC_i^p\|_F^2 \right)
			+ \|\TC_i^p\|_F^2 +\| \QC_i^p\|_F^2 
			+ \|\PC_i^p\|_F^2 \right) \\
			& = \sum_{i=1}^3 a_i\Psi_i(\XC_i^{p+1})
			+ \tau \|\FC^{p+1}\|_{1} 
			+ \mathds{1}_{P_{\Omega} (\EC^{p+1})=\mathcal{O}}
			+ \mathds{1}_{L^{p+1}_i{L_i^{p+1}}^T = I_i} \\
			& \quad + \frac{\mu^p}{2} \sum_{i=1}^3
			\left( \left\| \XC^{p+1}_i \times_i{L^{p+1}_i}^T
			- \YC^{p+1}_i + \frac{\NC^p_i}{\mu^p} \right\|_F^2 + \left\|\YC_i^{p+1} - \ZC^{p+1}
			+ \frac{\WC^p_i}{\mu^p} \right\|_F^2 \right)\\
			&  \quad + \frac{\mu^p}{2} \left( 
			\left\| \FC^{p+1} - D\MC^{p+1} 
			+ \frac{\TC^{p+1}}{\mu^p} \right\|_F^2
			+ \left\|\MC^{p+1} - \ZC^{p+1}
			+ \frac{\QC^p}{\mu^p} \right\|_F^2  \right.\\
			& \hspace{16em} \left.
			+ \left\|\ZC^{p+1} + \EC^{p+1} - \BC
			+ \frac{\PC^p}{\mu^p} \right\|_F^2 \right).
		\end{aligned}
	\end{equation*} 
	Note that the left side is bounded. 
	Thus, every term on the right side is bounded. 
	Briefly, we denote the bounded term $\Psi_i(\XC_i^{p+1})$ as 
	$\Psi_i(\XC),i=1,2,3$. 
	From the definition of mode-$i$ 
	LogTNN, we note that there exists a constant $d$ 
	such that the following inequality
	\begin{equation*}
		0 \leq {\rm log}\left( 
		1+\sigma_k^2\left(
		\left(\hat{\XC}_i\right)_i^{(j)}
		\right)\right) \leq d
	\end{equation*}
	holds for $1\leq k\leq m_i, 1 \leq j\leq  n_i$. 
	Then we can deduce $\sigma_k^2\left(
	\left(\hat{\XC}_i\right)_i^{(j)}
	\right)\leq e^{d}-1$ 
	and get
	\begin{equation*}
		\|\XC\|_F^2\leq m_i(e^{d}-1).
	\end{equation*}
	Therefore, for $i=1,2,3, \{\XC_i^{p}\} $ is bounded. 
	
	Moreover, $L_i^{p}$ is an unitary matrix, thus 
	$\{\XC_i^{p}\times_i{L_i^{p}}^T \}$ and $\{L_i\}$ are also bounded.	
	
	Since
	\begin{equation*}
		\YC_i^{p+1}=\XC_i^{p+1}\times_i{L_i^{p+1}}^T-\frac{\NC_i^k-\NC_i^{p+1}}{\mu^p}, 
	\end{equation*}
	$\{\YC_i^p\}$ is bounded. From other equations in \eqref{multiplier}, we obtain the boundedness of $\{\ZC^p\}$,$\{\FC^p\}$,$\{\MC^p\}$,$\{\EC^p\}$ similarly. This completes the proof of (a).
	
	Using the first-order optimal condition for $\{\YC_i^{p+1}\}$ in $\eqref{subY}$, we have
	\begin{equation*}
		2\YC_i^{p+1}-2\YC_i^p=\XC_i^{p}\times_i{L_i^{p}}^T-\YC_i^p+\ZC^p-\YC_i^p+\frac{\NC_i^p-\WC_i^p}{\mu^p}.
	\end{equation*}
	Combining \eqref{multiplier}, we obtain
	\begin{equation*}
		2(\YC_i^{p+1}-\YC_i^p)=\frac{\NC_i^{p}-\NC_i^{p-1}}{\mu^{p-1}}+\frac{\WC_i^{p-1}-\WC_i^p}{\mu^{p-1}}+\frac{\NC_i^p-\WC_i^p}{\mu^p}.
	\end{equation*}
	Therefore $\{\YC_i^{p}\}$ is a Cauchy sequence.
	
	Similarly, based on \eqref{multiplier} and the boundedness of multipliers' sequences, we can proof that $\{\XC_i^p\}$,$\{L_i^p\}$,$\{\ZC^p\}$,$\{\FC^p\}$,$\{\MC^p\}$,$\{\EC^p\}$ are Cauchy sequences.
\end{proof}

\begin{theorem}
	Let $\{S^p\}_{p=1}^{\infty} =
	\{\{\YC_i^p\}_{i=1}^3, \{\XC_i^p\}_{i=1}^3,
	\{L_i^p\}_{i=1}^3, \ZC^p, \FC^p, \MC^p, \EC^p,
	\{\NC_i^p\}_{i=1}^3, \linebreak
	\{\WC_i^p\}_{i=1}^3,$  $\QC^p, \TC^p, 
	\PC^p\}_{p=1}^{\infty}$ is generated from  algorithm \ref{Alg}. If the multipliers sequences are bounded, then $\{S^p\}_{p=1}^{\infty}$ converges to the limit point $S^*$, which satisfies the following KKT conditions:	
	\begin{align}
		&\YC_i^*=\ZC^*, \XC^*_i\times_i {L_i^*}^T=\YC_i^*,\FC^*=D\MC^*,\label{eq1}\\
		&\MC^*=\ZC^*,\ZC^*+\EC^*=\BC,P_{\Omega}(\EC^*)=\mathcal{O},
		L_i^*{L_i^*}^T =I_i,\label{eq2}\\
		&\mathcal{O}\in a_i\partial \Psi_i(\XC_i^*)+
		\NC_i^*\times_iL_i^*,\label{eq3}\\
		&\mathcal{O}\in\tau\partial\|\FC^*\|_1+\TC^*,\label{eq4}	\\
		&\mathcal{O}\in\partial(\mathds{1}_{P_{\Omega}(\EC^*)=\mathcal{O}})+\PC^*\label{eq5},\\
		&\mathcal{O}\in\partial(\mathds{1}_{L_i^*{L_i^*}^T=I_i})+\NC_{i(i)}^*{\XC_{i(i)}^*}^T.\label{eq6}
	\end{align}
\end{theorem}

\begin{proof}
	From \eqref{multiplier} and the boundedness of $\|\NC_i^p-\NC_i^{p-1}\|_F^2$, we have
	\begin{equation*}
		\lim_{p\to\infty}\left(\XC_i^p-\YC_i^p\times_i L_i^p\right)=\lim_{p\to\infty}\frac{\NC_i^{p-1}-\NC_i^p}{\mu^{p-1}}=\mathcal{O}.
	\end{equation*}
	Then, $\XC^*_i\times_i {L_i^*}^T=\YC_i^*$. Similarly, we acquire $\YC_i^*=\ZC^*, \FC^*=D\MC^*$, $\MC^*=\ZC^*, \ZC^*+\EC^*=\BC$.
	Since the solution of $\EC^{p+1},$ we know that $P_{\Omega}(\EC^*)=\mathcal{O}.$ Besides, from \eqref{soluL}, we get that $L_i^*$ is unitary. Now we complete the proof of \eqref{eq1} and \eqref{eq2}.
	
	From \eqref{subX}, we can deduce
	\begin{equation*}
		\mathcal{O}\in a_i\partial \Psi_i(\XC_i^{p+1})+
		\mu^p(\XC_i^{p+1}-\YC_i^{p+1}\times_i L_i^p)+\NC_i^p\times_i L_i^p.
	\end{equation*}
	Then we proof \eqref{eq3}.
	
	Similaily, from \eqref{subF} and \eqref{subE}, we acquire \eqref{eq4} and \eqref{eq5} respectively.
	
	From \eqref{subL}, we get
	\begin{equation*}
		\begin{aligned}
			\mathcal{O}\in&\partial(\mathds{1}_{L_i^{p+1}{(L_i^{p+1})}^T =I_i})+\partial\left(\frac{\mu^p}{2}\left\|\XC_i^{p+1}\times_i {(L_i^{p+1})}^T-\YC_i^{p+1}+\frac{\NC_i^p}{\mu^p}\right\|_F^2\right)\\
			=&\partial(\mathds{1}_{L_i^{p+1}{(L_i^{p+1})}^T =I_i})+\partial\left(\frac{\mu^p}{2}\left\|{(L_{i}^{p+1})}^{T}\XC_{i(i)}^{p+1}-\YC_{i(i)}^{p+1}+\frac{\NC_{i(i)}^p}{\mu^p}
			\right\|_F^2\right)\\
			=&\partial(\mathds{1}_{L_i^{p+1}{(L_i^{p+1})}^T =I_i})+\mu^p\left({(L_{i}^{p+1})}^{T}\XC_{i(i)}^{p+1}-\YC_{i(i)}^{p+1}+\frac{\NC_{i(i)}^p}{\mu^p}\right){\XC_{i(i)}^{p+1}}^T\\
			=&\partial(\mathds{1}_{L_i^{p+1}{(L_i^{p+1})}^T =I_i})+\mu^p\left({(L_{i}^{p+1})}^{T}\XC_{i(i)}^{p+1}-\YC_{i(i)}^{p+1}\right){\XC_{i(i)}^{p+1}}^T+\NC_{i(i)}^p	{\XC_{i(i)}^{p+1}}^T.\\
		\end{aligned}
	\end{equation*}
	Since the limit of the second term is $\mathcal{O}$, \eqref{eq6} can be proved.
\end{proof}

\section{Numerical Experiments}

In the proposed NLTFNN model, we employ the Log-Determinant function in the nonconvex surrogate.
In order to compare the difference between convex and nonconvex functions, we replace the TFNNLog in the NLTFNN model with  TFNN, thereby obtaining a learnable transformed fibered nuclear norm (LTFNN) method.
The LTFNN model will be utilized for comparison with the NLTFNN model in the experiments.
Apart from that, we also compared the proposed NLTFNN with six other algorithms: the classic TNN method \cite{ZEA14}, 
the TNN method based on DCT (TNNDCT) \cite{LPW19}, 
the transformed TNN with TV regularization (TNTV) \cite{QBN21}, 
the convex approximation of tensor fibered rank (3DTNN) \cite{ZHZ20}, 
the nonconvex Log-Determinant-based 3DTNN (3DLogTNN) \cite{ZHZ20}, and
the self-adaptive learnable transforms with Schattern-$p$ norm (SALTS) \cite{WGF22}.

In the existing methods, the parameters are selected based on the original papers. 
In the NLTFNN and LTFNN models, we set 
the parameters empirically, i.e., 
the penalty parameter 
$\mu^0 = 10^{-4}$ and its maximum $\mu_{\textit{max}} = 10$, 
the regularization parameter $\tau$ is set between 
$[10^{-3}, 10^{-8}],$
$\rho$ is selected between $(1,1.2]$, 
the stopping criterion $\epsilon=10^{-8}$,
the maximum iteration $\eta=500$,
and  $a=(a_1,a_2,a_3)$ is set to be $\frac{(1,1,0.001)}{2.001}$ or $\frac{(1,1,1)}{3}$.

To confirm the effectiveness of the proposed NLTFNN method in the tensor completion task, we consider different types of datasets, including magnetic resonance image (MRI), multispectral images (MSI), hyperspectral images (HSI) and videos.
For each dataset, we consider a 3-order tensor with randomly sampled entries, 
and the sampling rate (SR) is set to be 10\%, 20\%, 30\%, 40\% and 50\%.

In addition to the intuitive visualization, three quantitative image quality metrics provide a numerical evaluation of the performance of all the algorithms, including peak signal-to-noise ratio (PSNR), structural similarity (SSIM) and relative square error (RSE).
The higher the values of PSNR and SSIM, the better recovered data performance. 
Conversely, a smaller value of RSE indicates superior completion performance.


\subsection{MRI Completion}

In this subsection, we perform experiments on an MRI\footnote{http://brainweb.bic.mni.mcgill.ca/brainweb/selection$\_$normal.html} data of size $181 \times 217 \times 40$.
As we can seen in Table \ref{MRI_t}, NLTFNN method gives  the best numerical results in all the cases. 
From the enlarged part of Fig. \ref{MRI_result}, we can see that the image recovered by NLTFNN is smoother and clearer than those recovered by other methods.
A comparison of the results produced by the LTFNN and NLTFNN models reveals that NLTFNN reconstructs more details and visually better, which shows the advantage of nonconvex function.
Fig. \ref{MRI_slice} displays the PSNR and SSIM values with respect to the number of slices. It is evident that NLTFNN achieves the highest PSNR and SSIM values for each slice.

\begin{table}[t]
	\centering
	\caption{Recovered results of MRI by different methods with different SR.}
	\label{MRI_t}
	\scalebox{0.75}{
		\begin{tabular}{c c c c c c c c c c c c}
			\hline
			\multirow{2}{*}{SR} & \multirow{2}{*}{metric} &
			\multirow{2}{*}{Sample} & \multirow{2}{*}{TNN} &
			TNN & \multirow{2}{*}{TNTV} & 
			\multirow{2}{*}{3DTNN} & 3DLog & \multirow{2}{*}{SALTS} & LTFNN &  NLTFNN \\
			\multirow{2}{*}{} & \multirow{2}{*}{} & 
			\multirow{2}{*}{} & \multirow{2}{*}{} & 
			DCT & \multirow{2}{*}{} & 
			\multirow{2}{*}{} &  TNN & \multirow{2}{*}{} 
			& (Ours) & (Ours) \\
			\hline
			\multirow{3}{*}{10\% } & PSNR & 9.8857 & 24.1926 & 25.1658	& 26.4000 & 26.3164	& 26.7243 & 29.9008  &
			26.4264	& $\bm{30.0301}$\\
			\multirow{3}{*}{    } & SSIM & 0.2198 & 0.6706	&  0.7149	& 0.8366 &  0.7728	&  0.8167	& 0.8722 &
			0.8414	 & $\bm{0.9218}$\\
			\multirow{3}{*}{    } & RSE	 & 0.9485 & 0.1827	&  0.1633 & 0.1417	&  0.1431 & 0.1365	& 0.0958 &
			0.1413	&  $\bm{0.0933}$\\
			\hline
			\multirow{3}{*}{20\% } & PSNR & 10.3945	& 27.6075 &	 28.9064	& 30.3418 & 30.1375	& 32.7083 &  33.3647
			& 30.7049 &$\bm{34.3312}$ \\
			\multirow{3}{*}{    } & SSIM & 0.2564 & 0.8114	&  0.8551	&  0.9275	&  0.8983	&  0.9415 & 0.9318
			&  0.9337	&  $\bm{0.9614}$\\
			\multirow{3}{*}{    } & RSE	 & 0.8946 & 0.1233	& 0.1062 & 0.0900	&  0.0921 & 0.0685	 & 0.0643 &
			0.0863 & $\bm{0.0569}$ \\
			\hline
			\multirow{3}{*}{30\%} & PSNR & 10.9708 & 30.3649 & 31.8121	& 33.1654	& 32.6371	& 36.0667 &	35.8792  & 33.7149	& $\bm{36.8226}$\\
			\multirow{3}{*}{    } & SSIM & 0.2978 &  0.8869 &	  0.9201	&  0.9610	& 0.9424 & 0.9723	& 0.9584   &
			0.9661	&  $\bm{0.9756}$ \\
			\multirow{3}{*}{    } & RSE	 & 0.8371  & 0.0898	 & 0.0760	&  0.0650	&  0.0691	&  0.0466 & 0.0481 & 0.0610	 &  $\bm{0.0427}$  \\
			\hline
			\multirow{3}{*}{40\%} & PSNR & 11.6487 & 33.0028 & 34.4687	& 35.7129 & 34.7096	& 38.7785	& 38.0136  &
			36.3380	&  $\bm{39.5103}$\\
			\multirow{3}{*}{    } & SSIM & 0.3419 &  0.9331	 & 0.9547 & 0.9777	 &  0.9636	&  0.9849	& 0.9724   &
			0.9811	& 	$\bm{0.9874}$\\
			\multirow{3}{*}{    } & RSE	 &  0.7743 & 0.0663	&  0.0560 & 0.0485	& 0.0544 & 0.0341 & 0.0376 &
			0.0451	&  $\bm{0.0313}$  \\
			\hline
			\multirow{3}{*}{50\%} & PSNR & 12.4351	& 35.5250	& 36.9338	& 38.0220	& 36.4054 & 41.1048	&  39.8765 &
			38.6965	&  $\bm{41.6788}$ \\
			\multirow{3}{*}{    } & SSIM & 0.3879 &  0.9608	&  0.9736 & 0.9866	& 0.9750 & 0.9911	& 0.9805 &
			0.9889	& $\bm{0.9926}$ \\
			\multirow{3}{*}{    } & RSE	 & 0.7073 &  0.0496	&  0.0421 & 0.0372	 & 0.0448 & 0.0261 & 0.0303 &
			0.0344	&  $\bm{0.0244}$ \\
			\hline
		\end{tabular}
	}
\end{table}

\begin{figure}[htbp] 
	\centering		
	\setcounter{subfigure}{0}
	\subfloat[Original]{
		\includegraphics[width=0.18\linewidth]{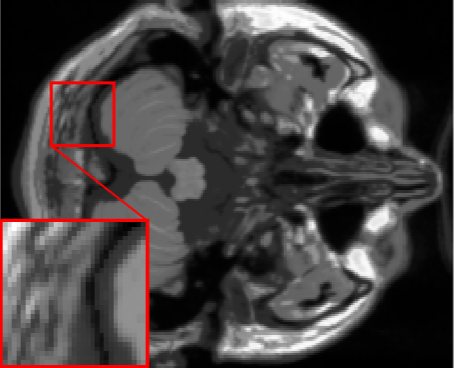}
	}
	\subfloat[Sample]{
		\includegraphics[width=0.18\linewidth]{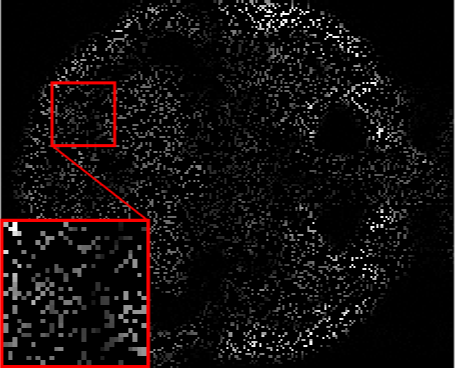}
	}
	\subfloat[TNN]{
		\includegraphics[width=0.18\linewidth]{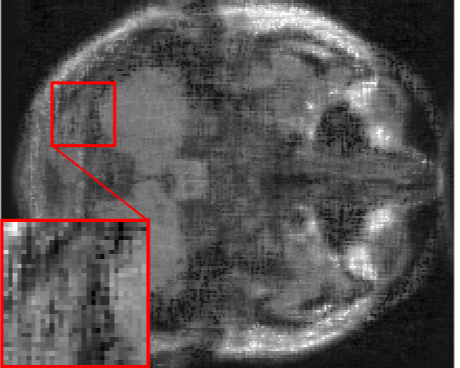}
	}
	\subfloat[TNNDCT]{
		\includegraphics[width=0.18\linewidth]{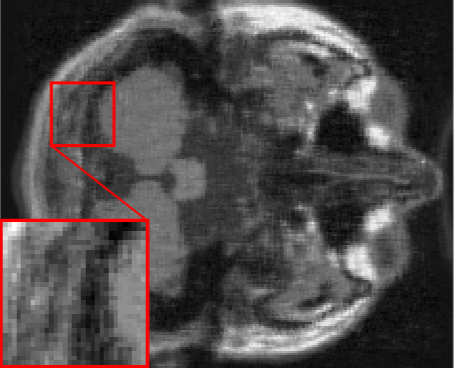}
	}
	\subfloat[TNTV]{
		\includegraphics[width=0.18\linewidth]{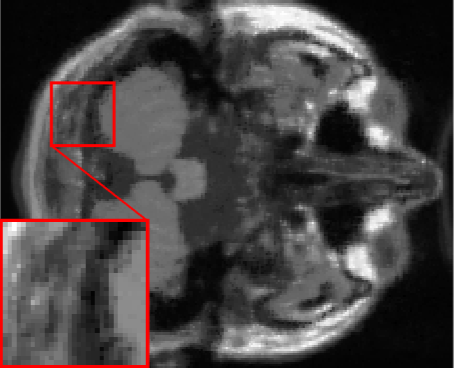}
	}
	\\
	\setcounter{subfigure}{5}
	\subfloat[3DTNN]{
		\includegraphics[width=0.18\linewidth]{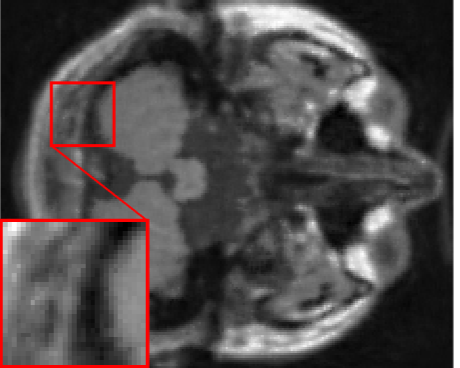}
	}
	\subfloat[3DLogTNN]{
		\includegraphics[width=0.18\linewidth]{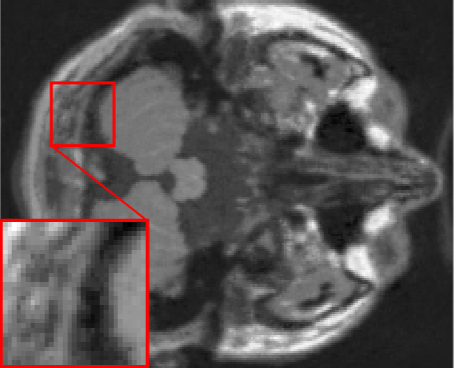}
	}
	\subfloat[SALTS]{
		\includegraphics[width=0.18\linewidth]{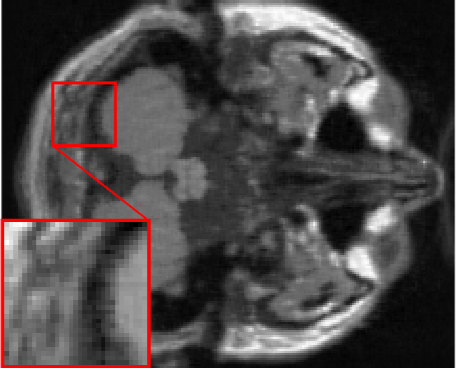}
	}
	\subfloat[LTFNN$_{\rm (Ours)}$]{
		\includegraphics[width=0.18\linewidth]{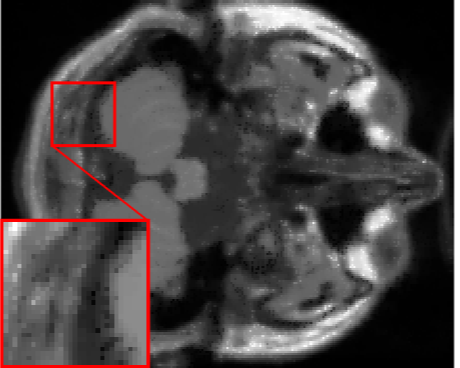}
	}
	\subfloat[NLTFNN$_{\rm (Ours)}$]{
		\includegraphics[width=0.18\linewidth]{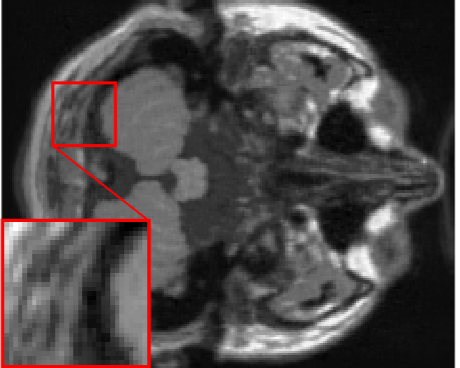}
	}
	\caption{The $1$st slice of MRI completion results with SR $= 20\%$. }
	\label{MRI_result}
\end{figure}

\begin{figure}[htbp] 
	\centering		
	\subfloat[]{
		\includegraphics[width=0.44\linewidth]{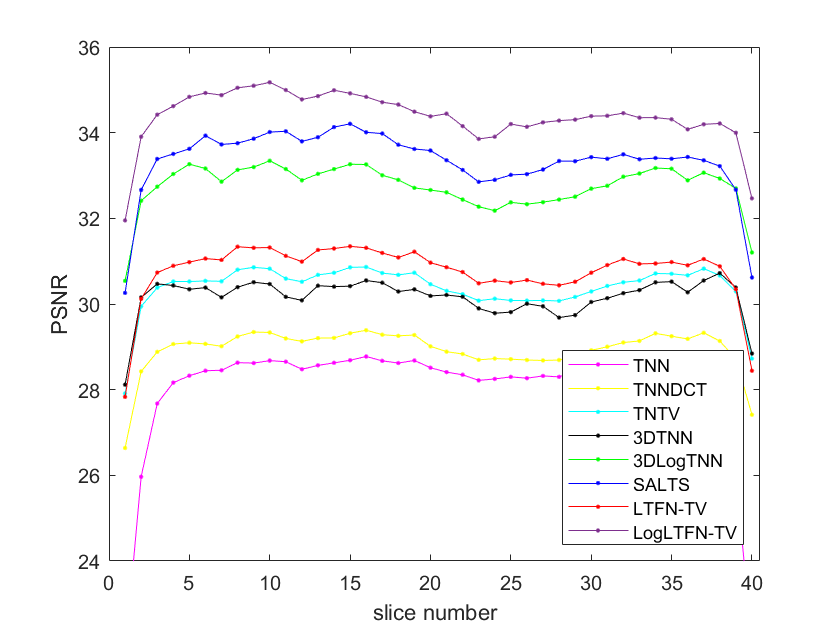}
	}
	\subfloat[]{
		\includegraphics[width=0.44\linewidth]{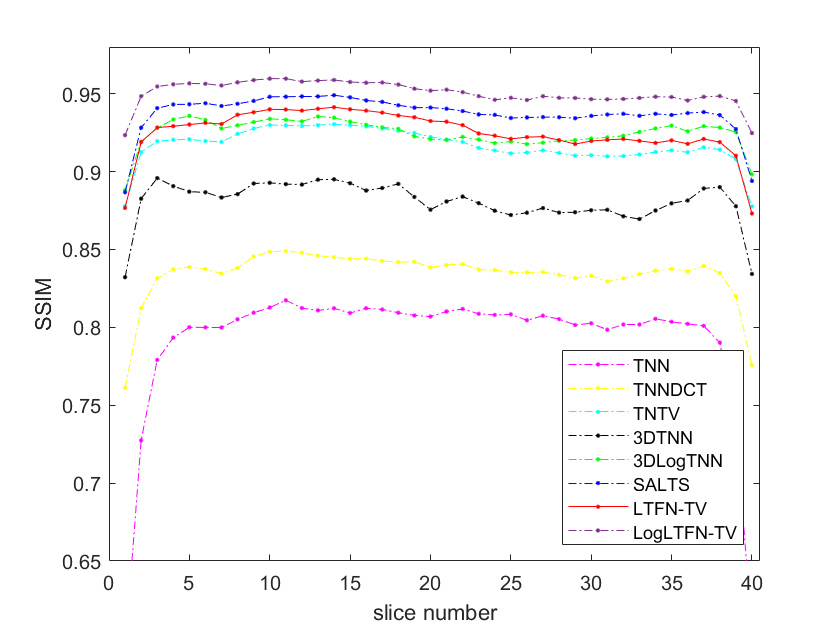}
	}
	\caption{The PSNR and SSIM values of different methods for MRI with SR $=20\%$. }
	\label{MRI_slice}
\end{figure}

\subsection{MSI Completion}

In this subsection, eight methods are compared for MSI\footnote{http://www.cs.columbia.edu/CAVE/databases/multispectral} recovery.
We firstly select the dataset \textit{cloth} of size $256 \times 256 \times 31$ and \textit{pompoms} of size $512 \times 512 \times 31$.
Then we test the methods with different SR and show the numerical results in Table \ref{cloth_t} and Table \ref{pompoms_t}, which indicate that our proposed NLTFNN method works best. 
In particular, as the SR increases, the difference between NLTFNN and 3DLogTNN becomes smaller, while the difference between NLTFNN and SALTS becomes larger.
Furthermore, Fig. \ref{pompoms_result} shows the visual comparison of the $20$th slice of MSI (\textit{pompoms}) completion results, where the images reconstructed by SALTS and NLTFNN have less noise than other algrithms.
This evidence suggests that the utilization of learnable transforms yields superior outcomes.
In Fig. \ref{cloth_slice}, the PSNR and SSIM values of each slice of \textit{cloth} are superior to other approaches.

\begin{table}[t]
	\centering
	\caption{\centering{Recovered results of MSI (\textit{cloth}) by different methods with different SR.}}
	\label{cloth_t}
	\scalebox{0.75}{
		\begin{tabular}{c c c c c c c c c c c c}
			\hline
			\multirow{2}{*}{SR} & \multirow{2}{*}{metric} &
			\multirow{2}{*}{Sample} & \multirow{2}{*}{TNN} &
			TNN & \multirow{2}{*}{TNTV} & 
			\multirow{2}{*}{3DTNN} & 3DLog & \multirow{2}{*}{SALTS} & LTFNN &  NLTFNN \\
			\multirow{2}{*}{} & \multirow{2}{*}{} & 
			\multirow{2}{*}{} & \multirow{2}{*}{} & 
			DCT & \multirow{2}{*}{} & 
			\multirow{2}{*}{} &  TNN & \multirow{2}{*}{} 
			& (Ours) & (Ours) \\
			\hline
			\multirow{3}{*}{10\% } & PSNR & 13.2948	& 25.9233	& 27.0751 & 25.5907	& 29.0002 & 29.9940	& 30.9833  & 25.5757 & $\bm{31.7470}$  \\
			\multirow{3}{*}{    } & SSIM & 0.0715 &	0.7551 &  0.7971 & 0.7425	& 0.8855 & 0.9151 &  0.8829   &
			0.7443	&   $\bm{0.9419}$\\
			\multirow{3}{*}{    } & RSE	 & 0.9487 & 0.2217 & 0.1941 & 0.2303	& 0.1556 & 0.1387 &	 0.1303 & 0.2307	&   $\bm{0.1134}$\\
			\hline
			\multirow{3}{*}{20\% } & PSNR & 13.8108	& 30.2017	& 31.8102 & 29.5139	 & 33.7327 & 36.6383 & 36.3176 & 30.0114 & $\bm{37.203}$ \\
			\multirow{3}{*}{    } & SSIM & 0.1355 &	  0.8963 & 0.9251 & 0.8983 &  0.9594 & 0.9791 & 0.9564    &
			0.9083	&   $\bm{0.9790}$\\
			\multirow{3}{*}{    } & RSE	 & 0.8940 & 0.1355 &  0.1126 & 0.1466	& 0.0902 &  0.0646 & 0.0745 &	  0.1385 &  $\bm{0.0605}$ \\
			\hline
			\multirow{3}{*}{30\%} & PSNR & 14.3876 & 33.0718	& 35.2512 & 33.0237	& 36.6552 & 41.0920 & 40.1376 & 33.8469	& $\bm{41.7787}$ \\
			\multirow{3}{*}{    } & SSIM &  0.2043 &	  0.9430 & 0.9646 & 0.9546	 &  0.9780	&  0.9921 &	 0.9778  & 0.9614	 &  $\bm{0.9924}$\\
			\multirow{3}{*}{    } & RSE	 & 0.8366 & 0.0973 &  0.0757 & 0.0979	 & 0.0644 & 0.0387 & 0.0505 &  0.0890	&   $\bm{0.0357}$  \\
			\hline
			\multirow{3}{*}{40\%} & PSNR & 15.0549	& 35.5169	& 38.1718 & 36.2458	& 39.3285 & 44.7202	& 42.9104 & 37.4157	&   $\bm{45.3932}$\\
			\multirow{3}{*}{    } & SSIM & 0.2755 &	  0.9660 & 0.9813 & 0.9778	& 0.9874 & 0.9964	&  0.9870 &
			0.9825	&  $\bm{0.9966 }$\\
			\multirow{3}{*}{    } & RSE	 & 0.7747 &  0.0735 &  0.0541 & 0.0675	 & 0.0474 & 0.0255 & 0.0371 &  0.0590	 & $\bm{ 0.0236}$  \\
			\hline
			\multirow{3}{*}{50\%} & PSNR & 15.8493 & 37.7489	& 40.7872 & 39.2273	& 41.5380 & 47.7902 & 45.1151 & 41.0284	&  $\bm{48.2910}$ \\
			\multirow{3}{*}{    } & SSIM & 0.3495 &	  0.9787 & 0.9896 & 0.9884	 & 0.9923 & 0.9982 & 0.9918 &
			0.9920	&   $\bm{0.9982}$ \\
			\multirow{3}{*}{    } & RSE	 & 0.7070 &	  0.0568 & 0.0400 & 0.0479 & 0.0367 & 0.0179 &  0.0285 &
			0.0389 &  $\bm{0.0169}$ \\
			\hline
		\end{tabular}
	}
\end{table}


\begin{table}[t]
	\centering
	\caption{\centering{Recovered results of MSI (\textit{pompoms}) by different methods with different SR.}}
	\label{pompoms_t}
	\scalebox{0.75}{
		\begin{tabular}{c c c c c c c c c c c c}
			\hline
			\multirow{2}{*}{SR} & \multirow{2}{*}{metric} &
			\multirow{2}{*}{Sample} & \multirow{2}{*}{TNN} &
			TNN & \multirow{2}{*}{TNTV} & 
			\multirow{2}{*}{3DTNN} & 3DLog & \multirow{2}{*}{SALTS} & LTFNN &  NLTFNN \\
			\multirow{2}{*}{} & \multirow{2}{*}{} & 
			\multirow{2}{*}{} & \multirow{2}{*}{} & 
			DCT & \multirow{2}{*}{} & 
			\multirow{2}{*}{} &  TNN & \multirow{2}{*}{} 
			& (Ours) & (Ours) \\
			\hline
			\multirow{3}{*}{10\% } & PSNR & 12.8811	& 32.2644	& 35.5106 & 37.0725	& 37.6398 & 39.9217 & 42.2809  & 37.8310 & $\bm{43.1315}$  \\
			\multirow{3}{*}{    } & SSIM & 0.1006 & 0.8745	&  0.9299 & 0.9569	& 0.9578 & 0.9767 & 0.9705   &
			0.9612	&   $\bm{0.9827}$\\
			\multirow{3}{*}{    } & RSE	 & 0.9486 & 0.1018	&  0.0701 & 0.0585	&  0.0548 & 0.0422	& 0.0326 &
			0.0537	&   $\bm{0.0291}$\\
			\hline
			\multirow{3}{*}{20\% } & PSNR & 13.3929	& 37.9779	& 41.0731 & 41.5036	& 42.1425 & 46.6968	 & 45.4602 & 42.3591	& $\bm{48.8495}$ \\
			\multirow{3}{*}{    } & SSIM & 0.1587 & 0.9560 &  0.9774 & 0.9811	 &  0.9826	&  0.9931 & 0.9843 &
			0.9830	& $\bm{ 0.9954}$\\
			\multirow{3}{*}{    } & RSE	 & 0.8943 & 0.0528	& 0.0369 & 0.0352	& 0.0327 & 0.0193 & 0.0225 &
			0.0319	& $\bm{0.0151}$ \\
			\hline
			\multirow{3}{*}{30\%} & PSNR & 13.9735	& 41.7964	& 44.7940 & 44.6305	&  44.8185	& 50.4835 & 47.2009 & 45.3227 & $\bm{51.9813}$\\
			\multirow{3}{*}{    } & SSIM & 0.2133 & 0.9797	& 0.9897	& 0.9900 & 0.9899 & 0.9969 & 0.9889 &
			0.9908 & $\bm{ 0.9978}$ \\
			\multirow{3}{*}{    } & RSE	 & 0.8365 & 0.0340	&  0.0241 & 0.0245 & 0.0240 & 0.0125 &  0.0184 &
			0.0226 &  $\bm{0.0105}$  \\
			\hline
			\multirow{3}{*}{40\%} & PSNR & 14.6399 & 44.9546	& 47.7371 & 47.2721	& 46.8600 & 53.1677	& 48.4506 &
			47.8080	& $\bm{54.0588}$\\
			\multirow{3}{*}{    } & SSIM & 0.2658 &	  0.9895 &  0.9945 & 0.9943	 &  0.9934	& 0.9983 & 0.9914 &
			0.9946	&  $\bm{0.9986}$\\
			\multirow{3}{*}{    } & RSE	 & 0.7747 &	  0.0236	& 0.0172 & 0.0181 &  0.0190	&  0.0092 & 0.0159 &
			0.0170	&  $\bm{0.0083}$  \\
			\hline
			\multirow{3}{*}{50\%} & PSNR & 15.4327	& 47.6142	& 50.1774 & 49.6639	& 48.7591 & 55.2138	& 49.4564 &
			50.1270	 &  $\bm{55.7230}$ \\
			\multirow{3}{*}{    } & SSIM & 0.3183 & 0.9941	& 0.9968 & 0.9966 & 0.9956 & 0.9989 & 0.9930 &
			0.9967	&  $\bm{0.9991}$ \\
			\multirow{3}{*}{    } & RSE	 & 0.7071 & 0.0174	& 0.0129 & 0.0137	 & 0.0152 & 0.0073	& 0.0141 &
			0.0130	&  $\bm{0.0068}$ \\
			\hline
		\end{tabular}
	}
\end{table}

\begin{figure}[htbp] 
	\centering		
	\setcounter{subfigure}{0}
	\subfloat[Original]{
		\includegraphics[width=0.18\linewidth]{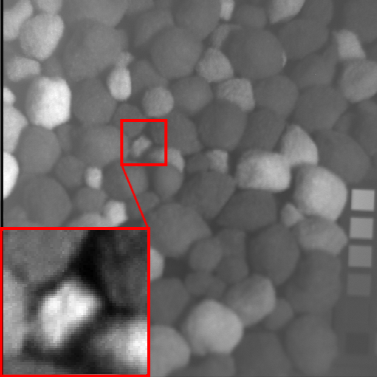}
	}
	\subfloat[Sample]{
		\includegraphics[width=0.18\linewidth]{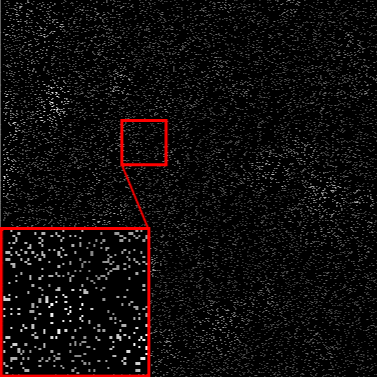}
	}
	\subfloat[TNN]{
		\includegraphics[width=0.18\linewidth]{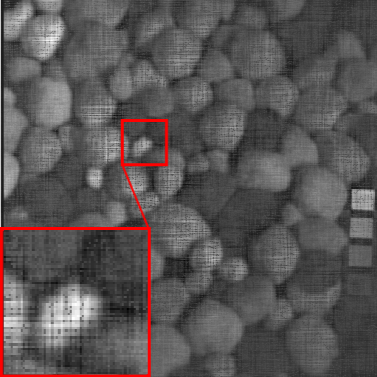}
	}
	\subfloat[TNNDCT]{
		\includegraphics[width=0.18\linewidth]{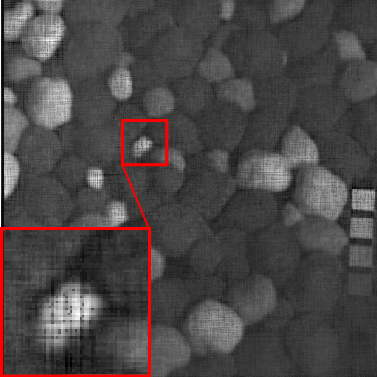}
	}
	\subfloat[TNTV]{
		\includegraphics[width=0.18\linewidth]{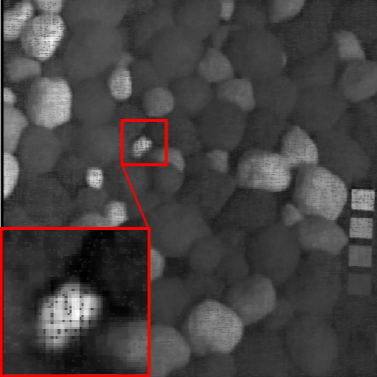}
	}
	\\
	\setcounter{subfigure}{5}
	\subfloat[3DTNN]{
		\includegraphics[width=0.18\linewidth]{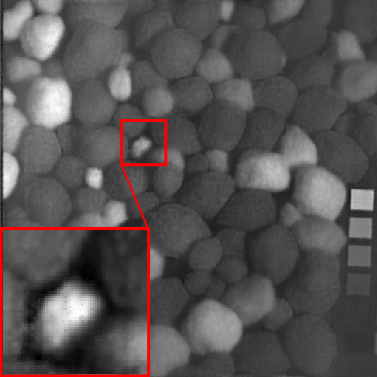}
	}
	\subfloat[3DLogTNN]{
		\includegraphics[width=0.18\linewidth]{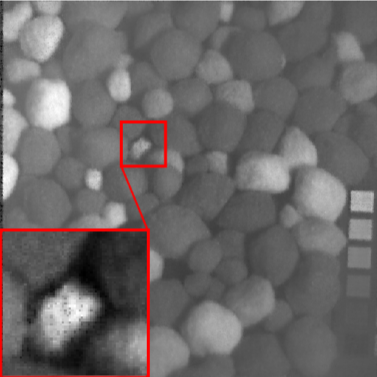}
	}
	\subfloat[SALTS]{
		\includegraphics[width=0.18\linewidth]{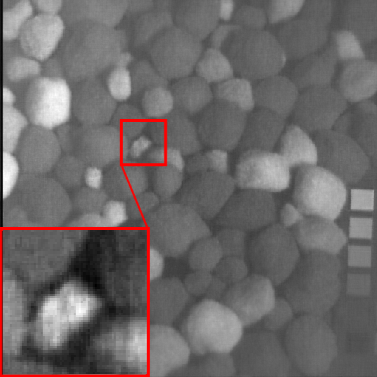}
	}
	\subfloat[LTFNN$_{\rm (Ours)}$]{
		\includegraphics[width=0.18\linewidth]{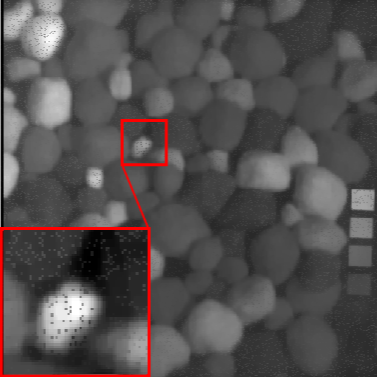}
	}
	\subfloat[NLTFNN$_{\rm (Ours)}$]{
		\includegraphics[width=0.18\linewidth]{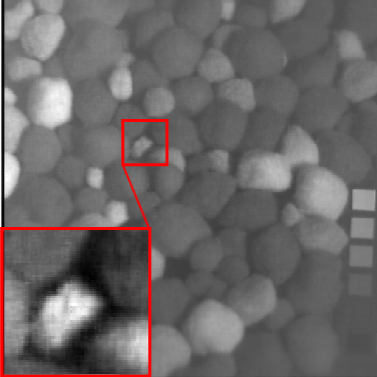}
	}
	\caption{The $20$th slice of MSI (\textit{pompoms}) completion results with SR $= 10\%$. }
	\label{pompoms_result}
\end{figure}

\begin{figure}[htbp] 
	\centering		
	\subfloat[]{
		\includegraphics[width=0.44\linewidth]{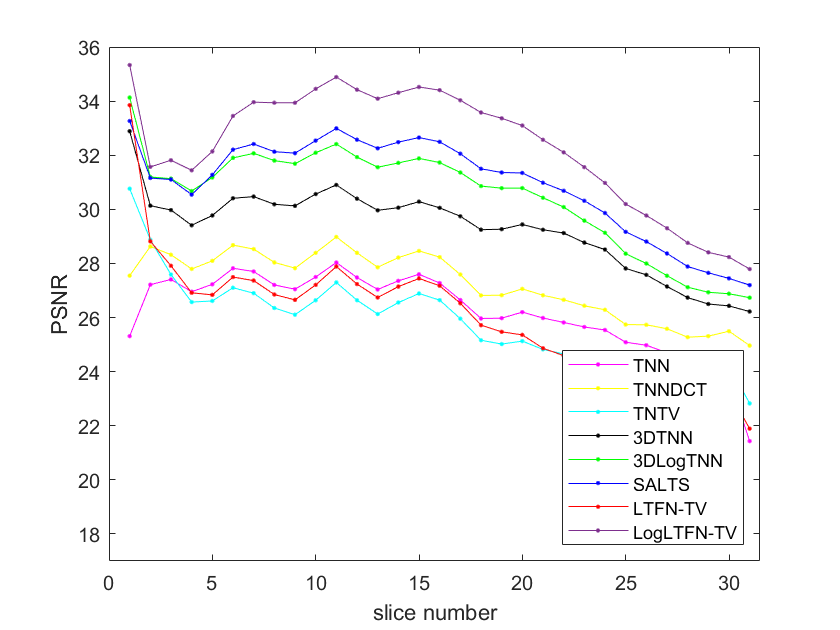}
	}
	\subfloat[]{
		\includegraphics[width=0.44\linewidth]{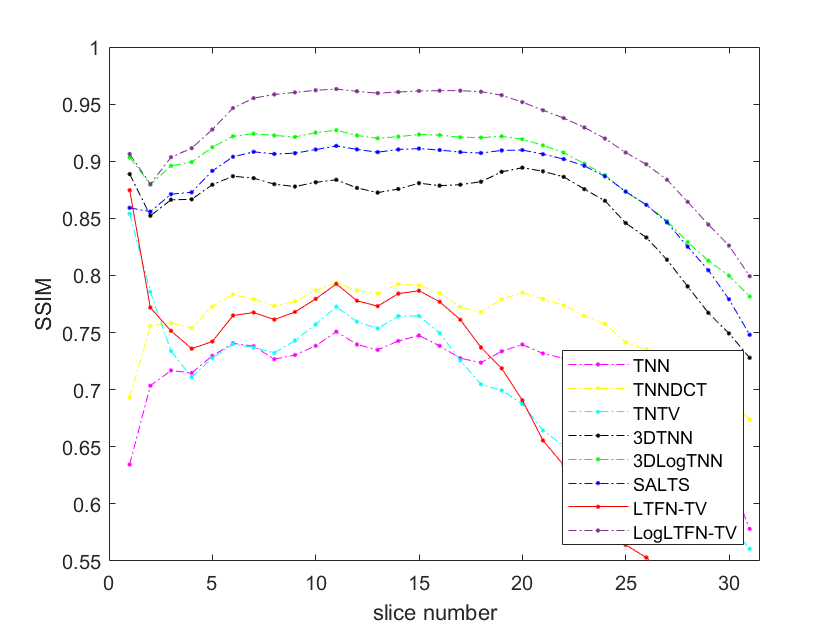}
	}
	\caption{The PSNR and SSIM values of different methods for MSI (\textit{cloth}) with SR $=10\%$. }
	\label{cloth_slice}
\end{figure}

\subsection{HSI Completion}

This subsection tests the HSI\footnote{https://personalpages.manchester.ac.uk/staff/d.h.foster//default.html} datasets \textit{scene6} and \textit{scene7} to validate the effectiveness of the proposed methods. 
To boost computational efficiency, we resized the datasets to $256 \times 256 \times 31$.
The recovered results are list in Table \ref{scene6_t} and Table \ref{scene7_t}, where NLTFNN has best performance. 
We give the visualization of the recovered results of \textit{scene6} in Fig. \ref{scene6_result}, which shows that the details of the image recovered by NLTFNN are most identical to the original image.
Moreover, Fig. \ref{scene7_slice} illustrates the numerical results of each slice of \textit{scene7}, demonstrating the superiority of NLTFNN.

\begin{table}[t]
	\centering
	\caption{\centering{Recovered results of HSI (\textit{scene6}) by different methods with different SR.}}
	\label{scene6_t}
	\scalebox{0.75}{
		\begin{tabular}{c c c c c c c c c c c c}
			\hline
			\multirow{2}{*}{SR} & \multirow{2}{*}{metric} &
			\multirow{2}{*}{Sample} & \multirow{2}{*}{TNN} &
			TNN & \multirow{2}{*}{TNTV} & 
			\multirow{2}{*}{3DTNN} & 3DLog & \multirow{2}{*}{SALTS} & LTFNN &  NLTFNN \\
			\multirow{2}{*}{} & \multirow{2}{*}{} & 
			\multirow{2}{*}{} & \multirow{2}{*}{} & 
			DCT & \multirow{2}{*}{} & 
			\multirow{2}{*}{} &  TNN & \multirow{2}{*}{} 
			& (Ours) & (Ours) \\
			\hline
			\multirow{3}{*}{10\% } & PSNR &  25.0928 & 39.2048	& 40.9338 & 39.7127	& 36.4496 & 42.3765 & 43.2676 & 39.0734	& $\bm{44.5069}$  \\
			\multirow{3}{*}{    } & SSIM & 0.2325 & 0.9484	&  0.9644 & 0.9565	& 0.8890 & 0.9734 & 0.9662  &
			0.9515	&  $\bm{0.9809}$\\
			\multirow{3}{*}{    } & RSE	 & 0.9483 & 0.1868 &  0.1531 & 0.1762 &  0.2565 & 0.1296	& 0.1301 &
			0.1896	&   $\bm{0.1014}$\\
			\hline
			\multirow{3}{*}{20\% } & PSNR & 25.6034	& 42.3665	& 44.9725 & 44.2163 & 41.0675 & 48.0585 & 46.3896  & 44.3902 &  $\bm{49.4450}$ \\
			\multirow{3}{*}{    } & SSIM & 0.3309 &	  0.9742 & 0.9858 & 0.9840	& 0.9609 & 0.9923 & 0.9818 &
			0.9848	&    $\bm{0.9937}$\\
			\multirow{3}{*}{    } & RSE	 & 0.8942 & 0.1298	&  0.0962	& 0.1049 & 0.1507 & 0.0674 & 0.0927 &
			0.1028	& $\bm{ 0.0575 }$ \\
			\hline
			\multirow{3}{*}{30\%} & PSNR & 26.1776 & 44.5761	& 47.7800 & 47.2208	& 43.5561 & 51.4354	& 47.9085 & 48.2597	&  $\bm{52.2420}$\\
			\multirow{3}{*}{    } & SSIM & 0.4169 & 0.9842	&  0.9926 & 0.9921	& 0.9785 & 0.9964	& 0.9867 &
			0.9935	&   $\bm{0.9966}$ \\
			\multirow{3}{*}{    } & RSE	 & 0.8370 & 0.1006	&  0.0696 & 0.0742	& 0.1132 & 0.0457 & 0.0765 &
			0.0659 &  $\bm{0.0416}$  \\
			\hline
			\multirow{3}{*}{40\%} & PSNR & 26.8522	& 46.9033	& 50.5919 & 50.1605	& 45.6538 &	 54.1374 & 48.9878 & 51.6159	&  $\bm{54.4957}$\\
			\multirow{3}{*}{    } & SSIM & 0.4936 &	  0.9901 & 0.9958 & 0.9957	& 0.9863 & 0.9979 &  0.9895 &
			0.9967	&   $\bm{0.9978}$\\
			\multirow{3}{*}{    } & RSE	 &  0.7744	& 0.0770	&  0.0503 & 0.0529 & 0.0889	 & 0.0335 &   0.0669 &
			0.0448 &   $\bm{0.0321}$  \\
			\hline
			\multirow{3}{*}{50\%} & PSNR & 27.6401	& 48.9498	& 52.7437 & 52.3789 & 47.0621 & 55.9495	& 49.9317  & 54.0492 &  $\bm{56.2385}$ \\
			\multirow{3}{*}{    } & SSIM & 0.5623	& 0.9938 &  0.9974 & 0.9974 & 0.9900 &  0.9986	&  0.9914  &
			0.9980	 &  $\bm{0.9986}$ \\
			\multirow{3}{*}{    } & RSE	 & 0.7073 &  0.0608	&  0.0393 &  0.0410 & 0.0756 &	 0.0272	&  0.0592 &
			0.0338	&   $\bm{0.0263}$ \\
			\hline
		\end{tabular}
	}
\end{table}

\begin{figure}[htbp] 
	\centering		
	\setcounter{subfigure}{0}
	\subfloat[Original]{
		\includegraphics[width=0.18\linewidth]{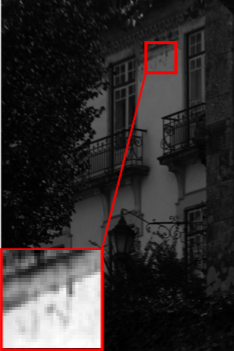}
	}
	\subfloat[Sample]{
		\includegraphics[width=0.18\linewidth]{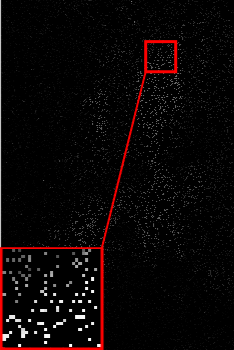}
	}
	\subfloat[TNN]{
		\includegraphics[width=0.18\linewidth]{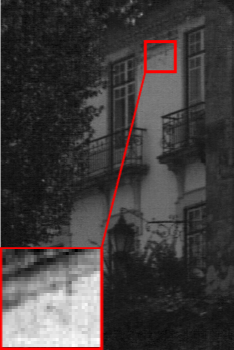}
	}
	\subfloat[TNNDCT]{
		\includegraphics[width=0.18\linewidth]{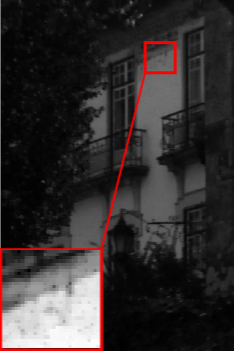}
	}
	\subfloat[TNTV]{
		\includegraphics[width=0.18\linewidth]{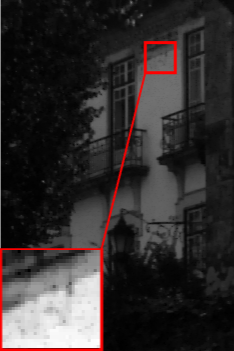}
	}
	\\
	\setcounter{subfigure}{5}
	\subfloat[3DTNN]{
		\includegraphics[width=0.18\linewidth]{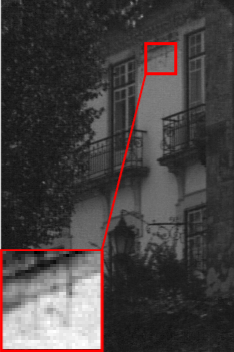}
	}
	\subfloat[3DLogTNN]{
		\includegraphics[width=0.18\linewidth]{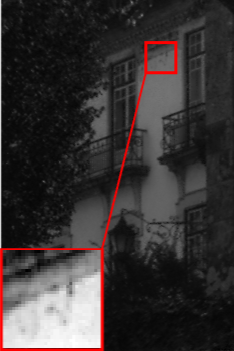}
	}
	\subfloat[SALTS]{
		\includegraphics[width=0.18\linewidth]{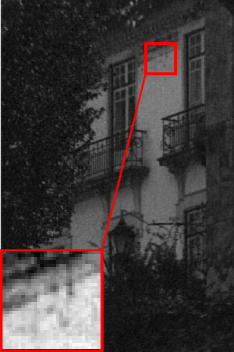}
	}
	\subfloat[LTFNN$_{\rm (Ours)}$]{
		\includegraphics[width=0.18\linewidth]{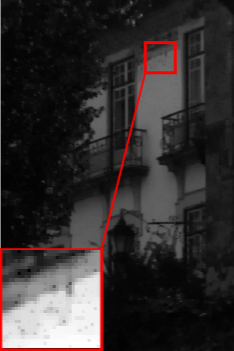}
	}
	\subfloat[NLTFNN$_{\rm (Ours)}$]{
		\includegraphics[width=0.18\linewidth]{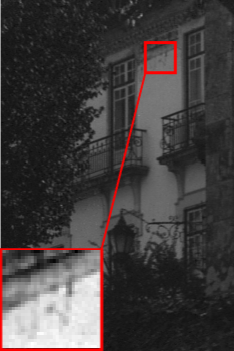}
	}
	\caption{The $9$th slice of HSI (\textit{scene6}) completion results with SR $= 10\%$. }
	\label{scene6_result}
\end{figure}

\begin{table}[t]
	\centering
	\caption{\centering{Recovered results of HSI (\textit{scene7}) by different methods with different SR.}}
	\label{scene7_t}
	\scalebox{0.75}{
		\begin{tabular}{c c c c c c c c c c c c}
			\hline
			\multirow{2}{*}{SR} & \multirow{2}{*}{metric} &
			\multirow{2}{*}{Sample} & \multirow{2}{*}{TNN} &
			TNN & \multirow{2}{*}{TNTV} & 
			\multirow{2}{*}{3DTNN} & 3DLog & \multirow{2}{*}{SALTS} & LTFNN &  NLTFNN \\
			\multirow{2}{*}{} & \multirow{2}{*}{} & 
			\multirow{2}{*}{} & \multirow{2}{*}{} & 
			DCT & \multirow{2}{*}{} & 
			\multirow{2}{*}{} &  TNN & \multirow{2}{*}{} 
			& (Ours) & (Ours) \\
			\hline
			\multirow{3}{*}{10\% } & PSNR & 15.3932	& 39.5143	& 41.2143 & 39.6929	& 35.8494 & 40.2315 & 43.5273  & 41.4016 & $\bm{43.8630}$  \\
			\multirow{3}{*}{    } & SSIM & 0.0975 &	  0.9633 & 0.9744 & 0.9687	& 0.9387 & 0.9744 & 0.9747 &
			0.9784	&   $\bm{0.9846}$\\
			\multirow{3}{*}{    } & RSE	 & 0.9487 & 0.0590	& 0.0485	& 0.0578 & 0.0900 & 0.0544	&  0.0394 &
			0.0475	&   $\bm{0.0358}$\\
			\hline
			\multirow{3}{*}{20\% } & PSNR & 15.9090	& 43.9860	& 45.6021 &	44.8187	& 39.7502 & 46.4044	& 46.4287 & 45.7616	&  $\bm{49.8029}$ \\
			\multirow{3}{*}{    } & SSIM & 0.1514 & 0.9852	&  0.9896 & 0.9885	& 0.9682 & 0.9912 & 0.9860 &
			0.9897	&   $\bm{0.9950}$\\
			\multirow{3}{*}{    } & RSE	 & 0.8940 & 0.0353	& 0.0293 & 0.0321	&  0.0574 & 0.0267 & 0.0282 &
			0.0288	&   $\bm{0.0181}$ \\
			\hline
			\multirow{3}{*}{30\%} & PSNR & 16.4835	& 46.7548	& 48.2326 & 47.8497	& 42.3314 & 50.0106	& 47.8376 & 49.2066	&  $\bm{52.0273}$\\
			\multirow{3}{*}{    } & SSIM & 0.1970 & 0.9915	&  0.9939 & 0.9937 &  0.9803 & 0.9955	& 0.9895 &
			0.9946	&  $\bm{0.9968}$ \\
			\multirow{3}{*}{    } & RSE	 & 0.8368 &	  0.0256	&  0.0216 & 0.0226 & 0.0427	&  0.0176 & 0.0239 &
			0.0193 &  $\bm{0.0140}$  \\
			\hline
			\multirow{3}{*}{40\%} & PSNR & 17.1544	& 49.0624	& 50.3835 & 50.2460	& 44.6623 & 52.5354 & 	48.8316 & 51.6011 &   $\bm{53.6385}$\\
			\multirow{3}{*}{    } & SSIM & 0.2394 &  0.9946 &  0.9960 & 0.9960	& 0.9876 & 0.9972 & 0.9914 &
			0.9966 &  $\bm{0.9977}$\\
			\multirow{3}{*}{    } & RSE	 & 0.7746 & 0.0197	&  0.0169 & 0.0172 & 0.0326 &  0.0132	&  0.0213 &
			0.0147 & $\bm{0.0116}$  \\
			\hline
			\multirow{3}{*}{50\%} & PSNR & 17.9493	& 50.9826	& 52.2138 & 52.1825	& 46.5790 & 54.3197 & 	49.6060 & 53.4978	&  $\bm{55.0581}$ \\
			\multirow{3}{*}{    } & SSIM & 0.2814 &  0.9964 &  0.9972	&  0.9973 & 0.9917	&  0.9981 &	 0.9927   & 0.9977 &   $\bm{0.9983}$ \\
			\multirow{3}{*}{    } & RSE	 & 0.7069 &  0.0158 &  0.0137 & 0.0137	&  0.0262 & 0.0107	&  0.0195 &
			0.0118	&   $\bm{0.0099}$ \\
			\hline
		\end{tabular}
	}
\end{table}


\begin{figure}[htbp] 
	\centering		
	\subfloat[]{
		\includegraphics[width=0.44\linewidth]{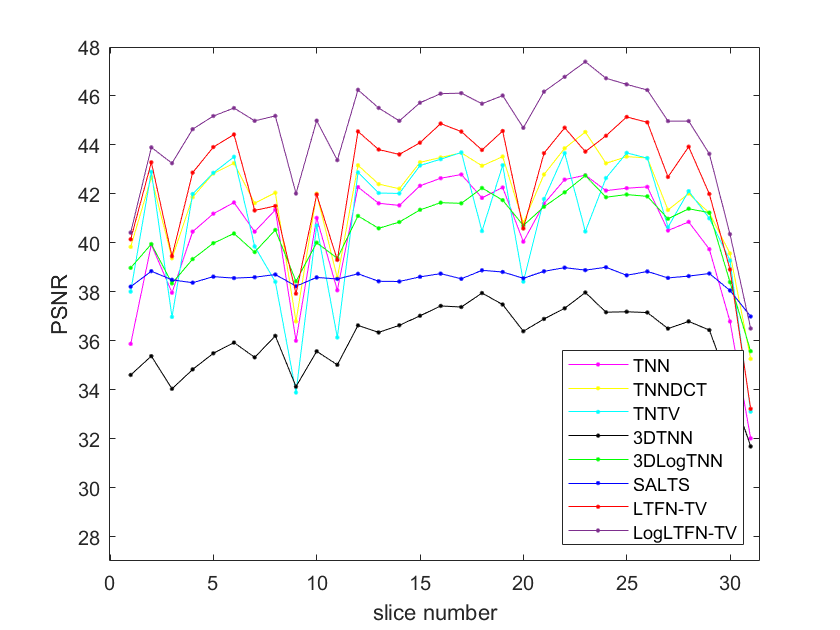}
	}
	\subfloat[]{
		\includegraphics[width=0.44\linewidth]{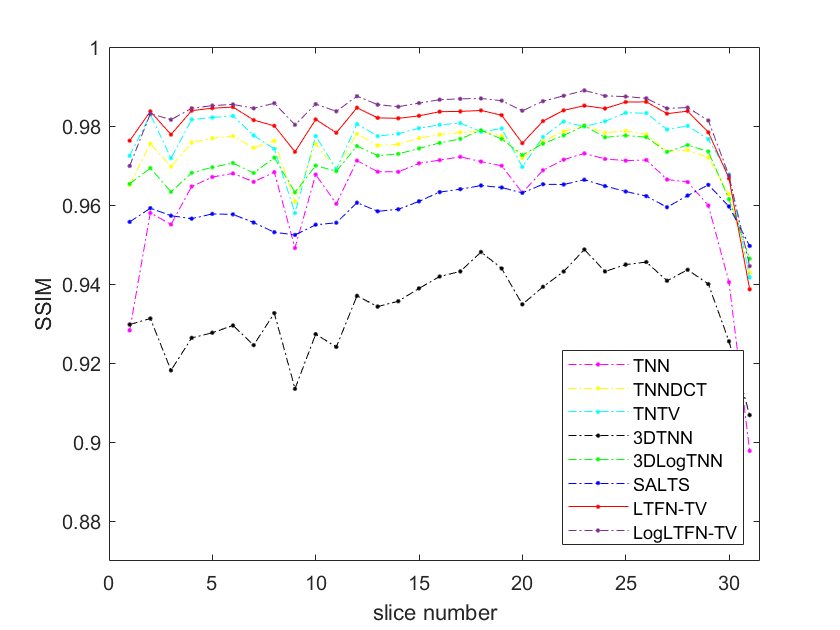}
	}
	\caption{The PSNR and SSIM values of different methods for HSI (\textit{scene7}) with SR $=10\%$. }
	\label{scene7_slice}
\end{figure}

\subsection{Video Completion}

In this subsection, we use eight algorithms to recover grayscale videos. 
First of all, we select three videos, \textit{Akiyo}\footnote{http://trace.eas.asu.edu/yuv/}, 
\textit{waterfall}\footnote{http://trace.eas.asu.edu/yuv/}
and \textit{road}\footnote{http://www.changedetection.net}.
For the improvement of the calculation time,
they are resized to 
$144 \times 176 \times 30$, 
$288 \times 352 \times 30$ 
and $158 \times 238 \times 24$, respectively.
The recovered results of videos are list in Table \ref{akiyo_t}, Table \ref{waterfall_t} and Table \ref{road_t}. 
It is worth mentioning that NLTFNN gives the most excellent results among all the videos, because their PSNR and SSIM values are the highest.
Additionally, we visualize the recovered results of \textit{waterfall} and \textit{road} in  Fig. \ref{waterfall_result} and Fig. \ref{road_result}, respectively.
Fig. \ref{waterfall_result} indicates that the videos recovered by SALTS and NLTFNN achieve satisfactory effects, which shows that the use of learnable transforms improves the effectiveness of recovery.
In Fig. \ref{road_result}, compared with 3DLogTNN and SALTS, our proposed methods produce clearer and smoother results, which shows that TV regularization helps to preserve details and maintain local smoothness.
Fig. \ref{akiyo_slice} compares the PSNR and SSIM values of each slice of \textit{Akiyo}, where NLTFNN gives the best in every frame.
Therefore, we claim that our proposed NLTFNN method yields a favorable performance in video recovery.

\begin{table}[t]
	\centering
	\caption{\centering{Recovered results of video (\textit{Akiyo}) by different methods with different SR.}}	
	\label{akiyo_t}
	\scalebox{0.75}{
		\begin{tabular}{c c c c c c c c c c c c}
			\hline
			\multirow{2}{*}{SR} & \multirow{2}{*}{metric} &
			\multirow{2}{*}{Sample} & \multirow{2}{*}{TNN} &
			TNN & \multirow{2}{*}{TNTV} & 
			\multirow{2}{*}{3DTNN} & 3DLog & \multirow{2}{*}{SALTS} & LTFNN &  NLTFNN \\
			\multirow{2}{*}{} & \multirow{2}{*}{} & 
			\multirow{2}{*}{} & \multirow{2}{*}{} & 
			DCT & \multirow{2}{*}{} & 
			\multirow{2}{*}{} &  TNN & \multirow{2}{*}{} 
			& (Ours) & (Ours) \\
			\hline
			\multirow{3}{*}{10\% } & PSNR & 7.5832	& 31.1920	& 31.6112 & 31.7595	& 29.0903 & 31.5222	& 34.8682  & 32.4460 &  $\bm{35.2485}$  \\
			\multirow{3}{*}{    } & SSIM & 0.0304 &  0.9359 &  0.9417	&  0.9570 & 0.8871 & 0.9347	&  0.9614 &
			0.9628	&    $\bm{0.9753}$\\
			\multirow{3}{*}{    } & RSE	 & 0.9486 & 0.0626	&  0.0597 & 0.0586	& 0.0797 & 0.0603 &	 0.0421 &
			0.0542 &   $\bm{0.0392}$\\
			\hline
			\multirow{3}{*}{20\% } & PSNR & 8.0926 & 34.7374	& 35.9171 & 36.0602	& 32.3059 & 37.2518	& 39.2065  & 37.8544 &  $\bm{40.2023}$ \\
			\multirow{3}{*}{    } & SSIM & 0.0493 & 0.9722 & 0.9791 & 0.9842	& 0.9591 & 0.9854 &	 0.9804 & 0.9890 &  $\bm{0.9923}$\\
			\multirow{3}{*}{    } & RSE	 & 0.8945 & 0.0416	&  0.0363	& 0.0357 &  0.0551 &  0.0312 & 	 0.0262 & 0.0291 &  $\bm{ 0.0222 }$ \\
			\hline
			\multirow{3}{*}{30\%} & PSNR &  8.6694	& 37.3219	& 38.9876 & 39.2168	& 35.1896 & 41.4010	& 41.7332 & 41.3029	&  $\bm{43.9832}$\\
			\multirow{3}{*}{    } & SSIM & 0.0687 & 0.9845 &  0.9895	&  0.9921 &  0.9783	&  0.9947 &  0.9865 &
			0.9948	&  $\bm{0.9960}$ \\
			\multirow{3}{*}{    } & RSE	 & 0.8370 & 0.0309	& 0.0255 & 0.0249	& 0.0395 & 0.0193 &  0.0197 &
			0.0195 & $\bm{ 0.0144}$  \\
			\hline
			\multirow{3}{*}{40\%} & PSNR & 9.3458 & 39.7557	& 41.9648	& 42.1872 & 37.4002	& 44.9311 & 43.5312 &
			44.6685	&  $\bm{47.3130}$\\
			\multirow{3}{*}{    } & SSIM & 0.0909 &  0.9907 & 0.9945 & 0.9957 & 0.9867 & 0.9974	&  0.9897 &
			0.9973	&  $\bm{0.9978}$\\
			\multirow{3}{*}{    } & RSE	 & 0.7743 &  0.0234	 & 0.0181	& 0.0177 & 0.0306	&  0.0129 & 0.0159 &
			0.0133	&   $\bm{0.0098}$  \\
			\hline
			\multirow{3}{*}{50\%} & PSNR & 10.1355 & 41.6650	& 44.3421 & 44.5144	& 39.3516 & 47.8180	& 45.0650 & 47.1966	& $\bm{ 49.6613 }$ \\
			\multirow{3}{*}{    } & SSIM &  0.1154	& 0.9940	& 0.9967 & 0.9974 &  0.9915 & 0.9985 & 0.9918 &
			0.9984	&   $\bm{0.9986}$ \\
			\multirow{3}{*}{    } & RSE	 & 0.7070 & 0.0187 &  0.0138	& 0.0135 & 0.0245 &  0.0092	&  0.0132 &
			0.0099	&  $\bm{ 0.0075}$ \\
			\hline
		\end{tabular}
	}
\end{table}


\begin{table}[t]
	\centering
	\caption{\centering{Recovered results of video (\textit{waterfall}) by different methods with different SR.}}
	\label{waterfall_t}
	\scalebox{0.75}{
		\begin{tabular}{c c c c c c c c c c c c}
			\hline
			\multirow{2}{*}{SR} & \multirow{2}{*}{metric} &
			\multirow{2}{*}{Sample} & \multirow{2}{*}{TNN} &
			TNN & \multirow{2}{*}{TNTV} & 
			\multirow{2}{*}{3DTNN} & 3DLog & \multirow{2}{*}{SALTS} & LTFNN &  NLTFNN \\
			\multirow{2}{*}{} & \multirow{2}{*}{} & 
			\multirow{2}{*}{} & \multirow{2}{*}{} & 
			DCT & \multirow{2}{*}{} & 
			\multirow{2}{*}{} &  TNN & \multirow{2}{*}{} 
			& (Ours) & (Ours) \\
			\hline
			\multirow{3}{*}{10\% } & PSNR & 8.2525 & 27.7850	& 28.2329 & 28.6091 & 28.6180 & 28.4764 & 30.1634 &  28.2464	&  $\bm{30.8462 }$  \\
			\multirow{3}{*}{    } & SSIM & 0.0221 &	  0.7147 & 0.7280 & 0.7601 &  0.7691 & 0.7769 & 0.8123    &
			0.7357 &  $\bm{0.8532}$\\
			\multirow{3}{*}{    } & RSE	 & 0.9486 & 0.1001	 &  0.0951	&  0.0910 & 0.0910 & 0.0924	&  0.0769 &
			0.0949	&   $\bm{0.0704}$\\
			\hline
			\multirow{3}{*}{20\% } & PSNR &  8.7641	& 29.9828	& 30.6678 & 31.3272	& 31.6602 & 31.5862	& 32.4466 &  31.3516 &  $\bm{32.7594}$ \\
			\multirow{3}{*}{    } & SSIM & 0.0350 & 0.8175 &  0.8358 & 0.8673	& 0.8728 & 0.8768	&  0.8835   &
			0.8697 &  $\bm{0.9016 }$\\
			\multirow{3}{*}{    } & RSE	 & 0.8944 & 0.0777	&  0.0718 & 0.0666	& 0.0641 & 0.0646 & 0.0591 &
			0.0664	&   $\bm{0.0565}$ \\
			\hline
			\multirow{3}{*}{30\%} & PSNR & 9.3463 & 31.7631	 & 32.5286	 & 33.3334 &  33.4428 & 33.7873 &	34.1678  & 33.5139	&  $\bm{34.4777}$\\
			\multirow{3}{*}{    } & SSIM & 0.0489 & 0.8743	&  0.8901 & 0.9162	&  0.9130 & 0.9235	&  0.9202     &
			0.9205	&   $\bm{0.9317}$ \\
			\multirow{3}{*}{    } & RSE	 & 0.8364 &  0.0633 &  0.0580 & 0.0529	& 0.0522 &	 0.0502	&  0.0485 &
			0.0518	&  $\bm{0.0463}$  \\
			\hline
			\multirow{3}{*}{40\%} & PSNR & 10.0176	& 33.4703	& 34.2205 & 35.0463	 & 34.8316	& 35.5831 &	35.8117  & 35.3640 &  $\bm{36.8387}$\\
			\multirow{3}{*}{    } & SSIM & 0.0648 & 0.9132 &  0.9251 & 0.9442 &  0.9362 & 0.9482	& 0.9435  &  
			0.9487	&   $\bm{0.9604}$\\
			\multirow{3}{*}{    } & RSE	 & 0.7742 &	  0.0520	&  0.0477 & 0.0434	& 0.0445 &	0.0408	& 0.0403 & 0.0418	&   $\bm{0.0353}$  \\
			\hline
			\multirow{3}{*}{50\%} & PSNR & 10.8052	& 35.3004	& 35.9957 & 36.7881	& 36.2146 & 37.4996	& 37.4681  & 37.2523 &  $\bm{39.1167}$ \\
			\multirow{3}{*}{    } & SSIM & 0.0827 &	  0.9424 & 0.9501 & 0.9632	 &  0.9534	&  0.9664 & 0.9601   & 
			0.9674	&   $\bm{0.9768}$ \\
			\multirow{3}{*}{    } & RSE	 & 0.7071 & 0.0421	& 0.0389 & 0.0355	&  0.0379 & 0.0327	&  0.0334 &
			0.0337	&   $\bm{0.0272}$ \\
			\hline
		\end{tabular}
	}
\end{table}

\begin{figure}[htbp] 
	\centering		
	\setcounter{subfigure}{0}
	\subfloat[Original]{
		\includegraphics[width=0.18\linewidth]{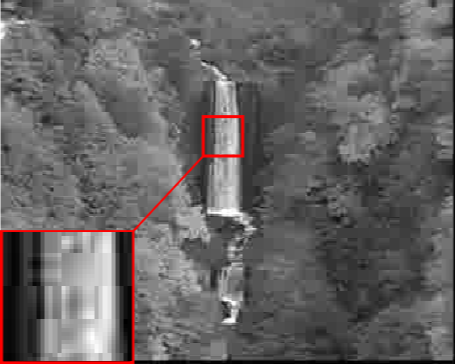}
	}
	\subfloat[Sample]{
		\includegraphics[width=0.18\linewidth]{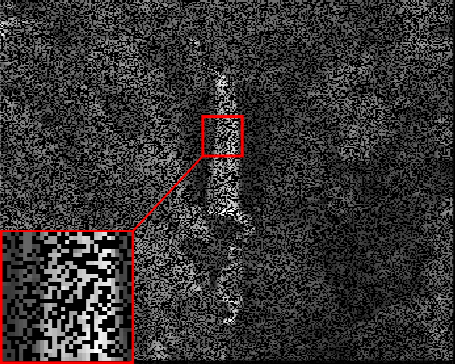}
	}
	\subfloat[TNN]{
		\includegraphics[width=0.18\linewidth]{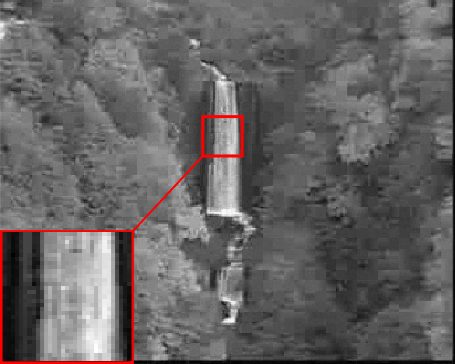}
	}
	\subfloat[TNNDCT]{
		\includegraphics[width=0.18\linewidth]{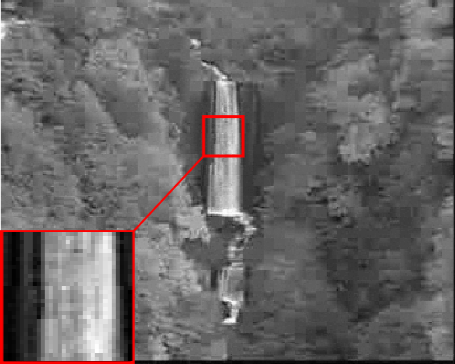}
	}
	\subfloat[TNTV]{
		\includegraphics[width=0.18\linewidth]{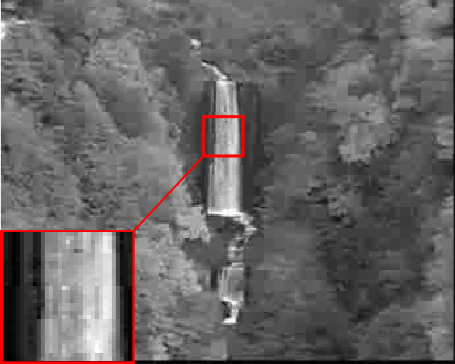}
	}
	\\
	\setcounter{subfigure}{5}
	\subfloat[3DTNN]{
		\includegraphics[width=0.18\linewidth]{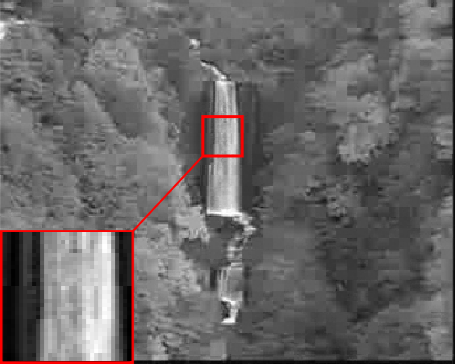}
	}
	\subfloat[3DLogTNN]{
		\includegraphics[width=0.18\linewidth]{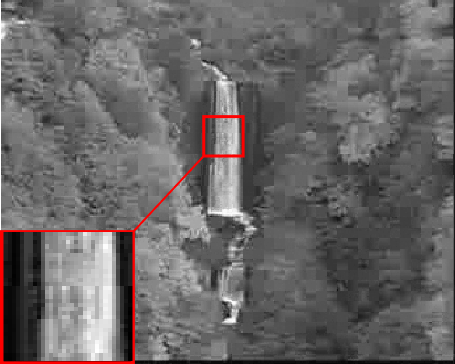}
	}
	\subfloat[SALTS]{
		\includegraphics[width=0.18\linewidth]{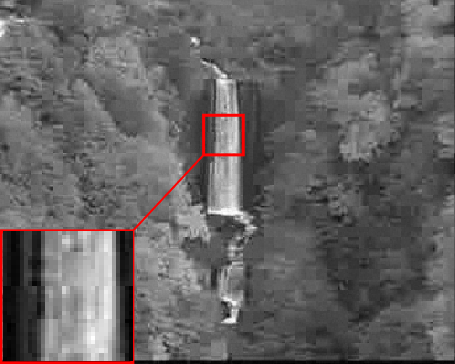}
	}
	\subfloat[LTFNN$_{\rm (Ours)}$]{
		\includegraphics[width=0.18\linewidth]{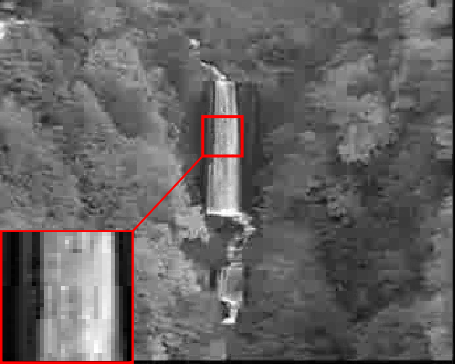}
	}
	\subfloat[NLTFNN$_{\rm (Ours)}$]{
		\includegraphics[width=0.18\linewidth]{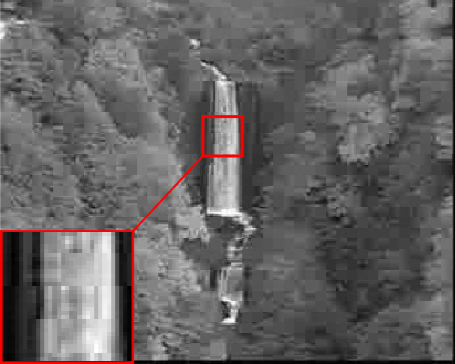}
	}
	\caption{The $25$th frame of video (\textit{waterfall}) completion results with SR $= 50\%$. }
	\label{waterfall_result}
\end{figure}

\begin{table}[t]
	\centering
	\caption{\centering{Recovered results of video (\textit{road}) by different methods with different SR.}}
	\label{road_t}
	\scalebox{0.75}{
		\begin{tabular}{c c c c c c c c c c c c}
			\hline
			\multirow{2}{*}{SR} & \multirow{2}{*}{metric} &
			\multirow{2}{*}{Sample} & \multirow{2}{*}{TNN} &
			TNN & \multirow{2}{*}{TNTV} & 
			\multirow{2}{*}{3DTNN} & 3DLog & \multirow{2}{*}{SALTS} & LTFNN &  NLTFNN \\
			\multirow{2}{*}{} & \multirow{2}{*}{} & 
			\multirow{2}{*}{} & \multirow{2}{*}{} & 
			DCT & \multirow{2}{*}{} & 
			\multirow{2}{*}{} &  TNN & \multirow{2}{*}{} 
			& (Ours) & (Ours) \\
			\hline
			\multirow{3}{*}{10\% } & PSNR & 6.6271	& 24.0809	& 24.3277 & 24.8606 & 24.8535 & 24.9577	& 26.1225 & 24.9259	& $\bm{27.0014 }$  \\
			\multirow{3}{*}{    } & SSIM & 0.0241 & 0.7059 &	0.7204 & 0.7520 & 0.8216 &  0.8241 &	0.8445 &  0.7696	&    $\bm{0.8444}$\\
			\multirow{3}{*}{    } & RSE	 & 0.9492 & 0.1273 &  0.1237 & 0.1163	& 0.1164 & 0.1150 &	  0.1009 &
			0.1155 &   $\bm{0.0909}$\\
			\hline
			\multirow{3}{*}{20\% } & PSNR & 7.1484	& 26.6714	& 26.9791 & 27.6698	& 27.8078 & 28.4412 & 28.6269 & 27.8264	& $\bm{29.6985}$ \\
			\multirow{3}{*}{    } & SSIM & 0.0464 & 0.8299 & 0.8413 & 0.8714	& 0.9090 &  0.9168 & 0.9072 &  0.8890 & $\bm{0.9234 }$\\
			\multirow{3}{*}{    } & RSE	 & 0.8939 &	 0.0944 & 0.0912 & 0.0842	 & 0.0829 & 0.0770 & 0.0755 &
			0.0827 &  $\bm{0.0666}$ \\
			\hline
			\multirow{3}{*}{30\%} & PSNR & 7.7245 & 28.6148	& 28.9304	& 29.7297 & 29.7194	& 30.6043 & 30.5144 &
			30.0144	&  $\bm{31.1664}$\\
			\multirow{3}{*}{    } & SSIM & 0.0696 &	 0.8857 &  0.8952 & 0.9228	& 0.9419 & 0.9498 & 0.9396   &
			0.9364	&   $\bm{0.9495}$ \\
			\multirow{3}{*}{    } & RSE	 & 0.8365 &	  0.0755 & 0.0728 & 0.0664 & 0.0665 &  0.0600	&   0.0608 &
			0.0643	&   $\bm{0.0563}$  \\
			\hline
			\multirow{3}{*}{40\%} & PSNR & 8.3869 & 30.3946	& 30.6238	& 31.4661 & 31.3654	& 32.6165 &	32.1836 & 31.8638	&  $\bm{33.1777}$\\
			\multirow{3}{*}{    } & SSIM & 0.0954 &	  0.9215 &  0.9282	&  0.9506 & 0.9599 &  0.9667 & 0.9578 &	  0.9606 &  $\bm{0.9687}$\\
			\multirow{3}{*}{    } & RSE	 &  0.7751 & 0.0615	&  0.0599 & 0.0544	& 0.0550 & 0.0476 & 0.0502 &
			0.0519 &  $\bm{0.0447}$  \\
			\hline
			\multirow{3}{*}{50\%} & PSNR & 9.1851 & 32.0474 & 32.2356	& 33.0579 & 32.7534	& 34.2839 & 33.8416  & 33.5625 &  $\bm{34.9272}$ \\
			\multirow{3}{*}{    } & SSIM &  0.1249	& 0.9455 & 0.9502 & 0.9673	& 0.9708 &  0.9767 & 0.9702   &
			0.9745 &   $\bm{0.9786}$ \\
			\multirow{3}{*}{    } & RSE	 & 0.7071 & 0.0509 &  0.0498 & 0.0453 & 0.0469 & 0.0393 & 0.0415 &
			0.0427 &  $\bm{0.0365}$ \\
			\hline
		\end{tabular}
	}
\end{table}

\begin{figure}[htbp] 
	\centering		
	\setcounter{subfigure}{0}
	\subfloat[Original]{
		\includegraphics[width=0.18\linewidth]{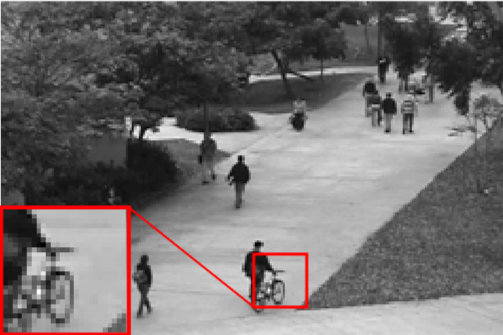}
	}
	\subfloat[Sample]{
		\includegraphics[width=0.18\linewidth]{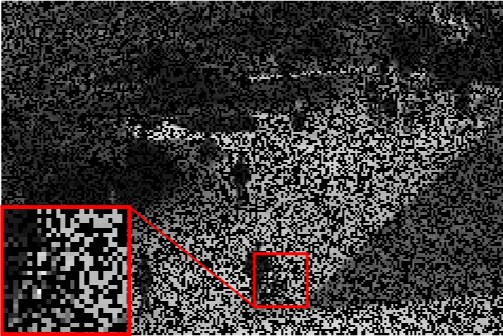}
	}
	\subfloat[TNN]{
		\includegraphics[width=0.18\linewidth]{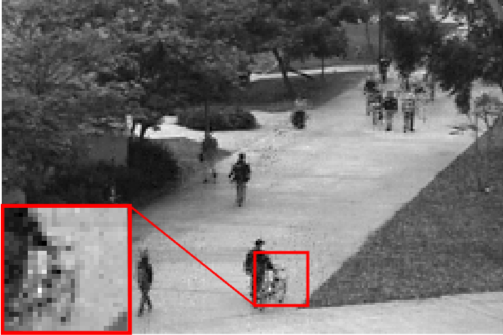}
	}
	\subfloat[TNNDCT]{
		\includegraphics[width=0.18\linewidth]{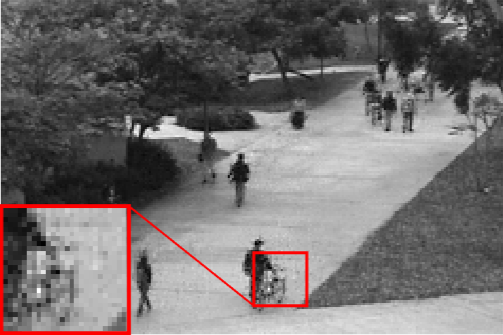}
	}
	\subfloat[TNTV]{
		\includegraphics[width=0.18\linewidth]{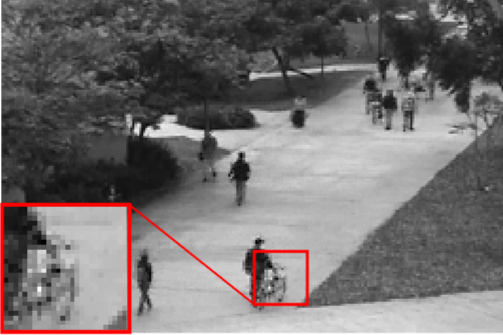}
	}
	\\
	\setcounter{subfigure}{5}
	\subfloat[3DTNN]{
		\includegraphics[width=0.18\linewidth]{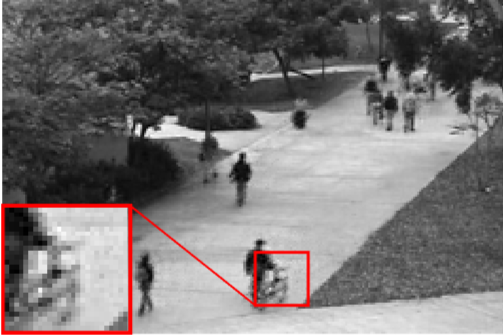}
	}
	\subfloat[3DLogTNN]{
		\includegraphics[width=0.18\linewidth]{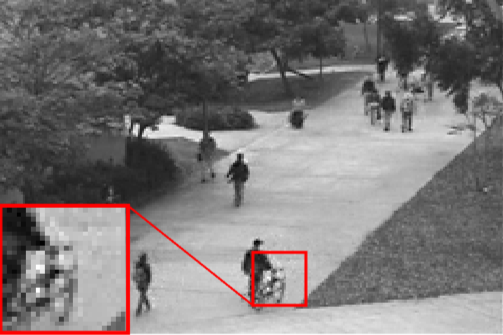}
	}
	\subfloat[SALTS]{
		\includegraphics[width=0.18\linewidth]{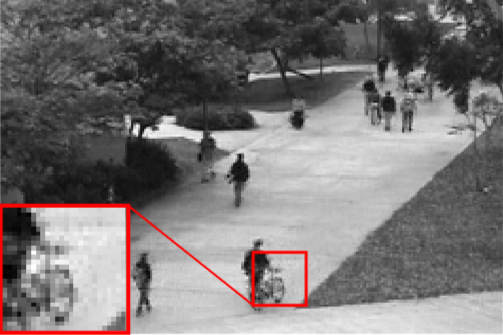}
	}
	\subfloat[LTFNN$_{\rm (Ours)}$]{
		\includegraphics[width=0.18\linewidth]{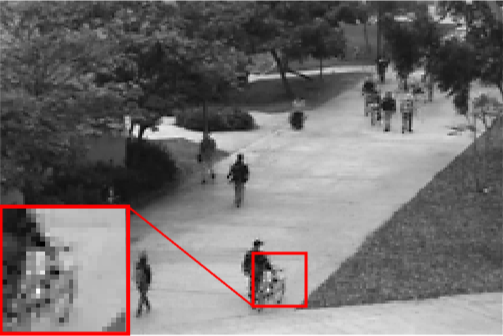}
	}
	\subfloat[NLTFNN$_{\rm (Ours)}$]{
		\includegraphics[width=0.18\linewidth]{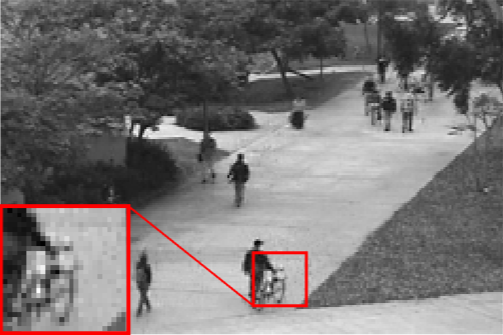}
	}
	\caption{The $2$nd frame of video (\textit{road}) completion results with SR $= 50\%$. }
	\label{road_result}
\end{figure}

\begin{figure}[htbp] 
	\centering		
	\subfloat[]{
		\includegraphics[width=0.44\linewidth]{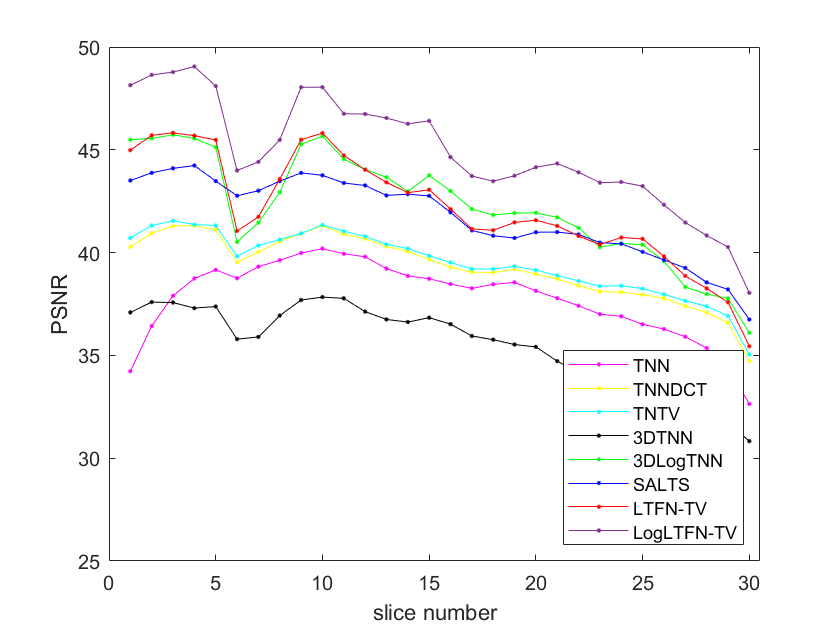}
	}
	\subfloat[]{
		\includegraphics[width=0.44\linewidth]{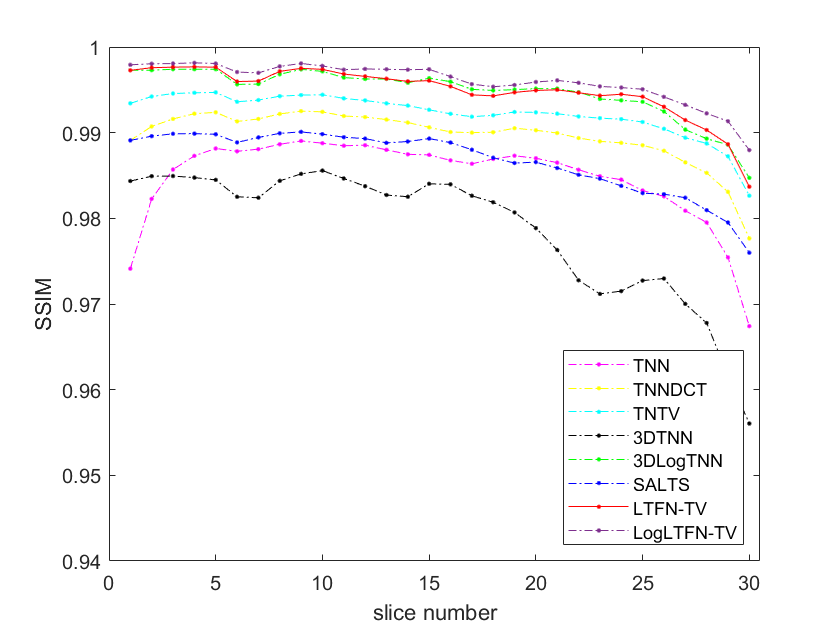}
	}
	\caption{The PSNR and SSIM values of different methods for video (\textit{Akiyo}) with SR $=30\%$. }
	\label{akiyo_slice}
\end{figure}

\section{Conclusion}

In this paper, we have proposed a novel LRTC model called NLTFNN, which utilizes the LTFNNLog and TV regularization.
Then the framework of ADMM is developed to solve NLTFNN model and the convergence theorems have also been presented.
Furthermore, the visual and numerical results in comparison with the other methods indicate the superior recovery performance of NLTFNN, demonstrating its flexibility to fully exploit low-rankness of the entire data and the ability to investigate internal smoothness.

\bibliographystyle{elsarticle-num} 
\bibliography{ref}

\begin{thebibliography}{10}
\expandafter\ifx\csname url\endcsname\relax
  \def\url#1{\texttt{#1}}\fi
\expandafter\ifx\csname urlprefix\endcsname\relax\def\urlprefix{URL }\fi
\expandafter\ifx\csname href\endcsname\relax
  \def\href#1#2{#2} \def\path#1{#1}\fi

\bibitem{GXC20}
X.~Gong, W.~Chen, J.~Chen, B.~Ai, Tensor denoising using low-rank tensor train
  decomposition, IEEE Signal Processing Letters 27 (2020) 1685--1689.

\bibitem{FLY22}
L.~Fu, J.~Yang, C.~Chen, C.~Zhang, Low-rank tensor approximation with local
  structure for multi-view intrinsic subspace clustering, Information Sciences
  606 (2022) 877--891.

\bibitem{ZYZ22}
Y.~Zhang, Y.~Zhang, W.~Wang, Learning linear and nonlinear low-rank structure
  in multi-task learning, IEEE Transactions on Knowledge and Data Engineering
  (2022).

\bibitem{QML20}
M.~Qin, Z.~Li, S.~Chen, Q.~Guan, J.~Zheng, Low-rank tensor completion and total
  variation minimization for color image inpainting, IEEE Access 8 (2020)
  53049--53061.

\bibitem{KB09}
T.~G. Kolda, B.~W. Bader, Tensor decompositions and applications, SIAM review
  51~(3) (2009) 455--500.

\bibitem{TUC66}
L.~R. Tucker, Some mathematical notes on three-mode factor analysis,
  Psychometrika 31~(3) (1966) 279--311.

\bibitem{KMM11}
M.~E. Kilmer, C.~D. Martin, Factorization strategies for third-order tensors,
  Linear Algebra and its Applications 435~(3) (2011) 641--658.

\bibitem{KBH13}
M.~E. Kilmer, K.~Braman, N.~Hao, R.~C. Hoover, Third-order tensors as operators
  on matrices: A theoretical and computational framework with applications in
  imaging, SIAM Journal on Matrix Analysis and Applications 34~(1) (2013)
  148--172.

\bibitem{ZHZ20}
Y.-B. Zheng, T.-Z. Huang, X.-L. Zhao, T.-X. Jiang, T.-H. Ma, T.-Y. Ji, Mixed
  noise removal in hyperspectral image via low-fibered-rank regularization,
  IEEE Transactions on Geoscience and Remote Sensing 58~(1) (2019) 734--749.

\bibitem{WGF22}
T.~Wu, B.~Gao, J.~Fan, J.~Xue, W.~L. Woo, Low-rank tensor completion based on
  self-adaptive learnable transforms, IEEE Transactions on Neural Networks and
  Learning Systems (2022).

\bibitem{XZN19}
W.~Xu, X.~Zhao, M.~Ng, A fast algorithm for cosine transform based tensor
  singular value decomposition (2019), Preprint arXiv (1902).

\bibitem{JNZ20}
T.-X. Jiang, M.~K. Ng, X.-L. Zhao, T.-Z. Huang, Framelet representation of
  tensor nuclear norm for third-order tensor completion, IEEE Transactions on
  Image Processing 29 (2020) 7233--7244.

\bibitem{JZZ21}
T.-X. Jiang, X.-L. Zhao, H.~Zhang, M.~K. Ng, Dictionary learning with low-rank
  coding coefficients for tensor completion, IEEE Transactions on Neural
  Networks and Learning Systems 34~(2) (2021) 932--946.

\bibitem{JHZ17}
T.-Y. Ji, T.-Z. Huang, X.-L. Zhao, T.-H. Ma, L.-J. Deng, A non-convex tensor
  rank approximation for tensor completion, Applied Mathematical Modelling 48
  (2017) 410--422.

\bibitem{GF20}
S.~Gao, Q.~Fan, Robust schatten-p norm based approach for tensor completion,
  Journal of Scientific Computing 82 (2020) 1--23.

\bibitem{XZJ19}
W.-H. Xu, X.-L. Zhao, T.-Y. Ji, J.-Q. Miao, T.-H. Ma, S.~Wang, T.-Z. Huang,
  Laplace function based nonconvex surrogate for low-rank tensor completion,
  Signal Processing: Image Communication 73 (2019) 62--69.

\bibitem{LHS15}
J.~Liu, T.-Z. Huang, I.~W. Selesnick, X.-G. Lv, P.-Y. Chen, Image restoration
  using total variation with overlapping group sparsity, Information Sciences
  295 (2015) 232--246.

\bibitem{LLL22}
Y.~Liu, J.~Liu, Z.~Long, C.~Zhu, Tensor computation for data analysis,
  Springer, 2022.

\bibitem{LZJ22}
B.-Z. Li, X.-L. Zhao, T.-Y. Ji, X.-J. Zhang, T.-Z. Huang, Nonlinear transform
  induced tensor nuclear norm for tensor completion, Journal of Scientific
  Computing 92~(3) (2022) 83.

\bibitem{YLL22}
M.~Yang, Q.~Luo, W.~Li, M.~Xiao, 3-d array image data completion by tensor
  decomposition and nonconvex regularization approach, IEEE Transactions on
  Signal Processing 70 (2022) 4291--4304.

\bibitem{CGW17}
Y.~Chen, Y.~Guo, Y.~Wang, D.~Wang, C.~Peng, G.~He, Denoising of hyperspectral
  images using nonconvex low rank matrix approximation, IEEE Transactions on
  Geoscience and Remote Sensing 55~(9) (2017) 5366--5380.

\bibitem{ZEA14}
Z.~Zhang, G.~Ely, S.~Aeron, N.~Hao, M.~Kilmer, Novel methods for multilinear
  data completion and de-noising based on tensor-svd, in: Proceedings of the
  IEEE conference on computer vision and pattern recognition, 2014, pp.
  3842--3849.

\bibitem{LPW19}
C.~Lu, X.~Peng, Y.~Wei, Low-rank tensor completion with a new tensor nuclear
  norm induced by invertible linear transforms, in: Proceedings of the IEEE/CVF
  conference on computer vision and pattern recognition, 2019, pp. 5996--6004.

\bibitem{QBN21}
D.~Qiu, M.~Bai, M.~K. Ng, X.~Zhang, Robust low-rank tensor completion via
  transformed tensor nuclear norm with total variation regularization,
  Neurocomputing 435 (2021) 197--215.

\end{thebibliography}





\end{document}